\documentclass[12pt]{amsart}

\usepackage{etex}
\usepackage{amsmath, amssymb}
\usepackage{array}
\usepackage[frame,cmtip,arrow,matrix,line,graph,curve]{xy}
\usepackage{graphpap, color, paralist}
\usepackage[mathscr]{eucal}
\usepackage[pdftex]{graphicx}
\usepackage[pdftex,colorlinks,backref=page,citecolor=blue]{hyperref}
\usepackage{tikz}
\usetikzlibrary{matrix,arrows,decorations.pathmorphing}

\newcolumntype{P}[1]{>{\centering\arraybackslash}p{#1}}

\allowdisplaybreaks

\newtheorem{theorem}{Theorem}[section]
\newtheorem{lemma}[theorem]{Lemma}
\newtheorem{proposition}[theorem]{Proposition}

\theoremstyle{definition}

\newtheorem{remark}[theorem]{Remark}

\newtheorem{example}[theorem]{Example}
\newtheorem{definition}[theorem]{Definition}

\newcommand{\CC}{\mathbb{C}}

\newcommand{\PP}{\mathbb{P}}
\newcommand{\ZZ}{\mathbb{Z}}
\newcommand{\rH}{\mathrm{H}}
\newcommand{\rM}{\mathrm{M}}

\newcommand{\cA}{\mathcal{A}}
\newcommand{\cB}{\mathcal{B}}
\newcommand{\cC}{\mathcal{C}}

\newcommand{\cJ}{\mathcal{J}}
\newcommand{\cK}{\mathcal{K}}
\newcommand{\cL}{\mathcal{L}}

\newcommand{\cN}{\mathcal{N}}
\newcommand{\cO}{\mathcal{O}}
\newcommand{\cP}{\mathcal{P}}
\newcommand{\cR}{\mathcal{R}}

\newcommand{\cT}{\mathcal{T}}
\newcommand{\cV}{\mathcal{V}}
\newcommand{\cW}{\mathcal{W}}
\newcommand{\cX}{\mathcal{X}}
\newcommand{\cY}{\mathcal{Y}}
\newcommand{\cZ}{\mathcal{Z}}
\newcommand{\bJ}{\mathbf{J}}
\newcommand{\bK}{\mathbf{K}}
\newcommand{\bM}{\mathbf{M}}
\newcommand{\bP}{\mathbf{P}}
\newcommand{\bU}{\mathbf{U}}

\newcommand{\bX}{\mathbf{X}}

\newcommand{\bp}{\mathbf{p}}

\newcommand{\NS}{\mathrm{NS}}

\newcommand{\SL}{\mathrm{SL}}

\newcommand{\git}{/\!/}

\tikzset{
    partial ellipse/.style args={#1:#2:#3}{
        insert path={+ (#1:#3) arc (#1:#2:#3)}
    }
}

\makeatletter
\def\l@subsection{\@tocline{2}{0pt}{2.5pc}{5pc}{}}
\makeatother

\author{Han-Bom Moon}
\address{Department of Mathematics, Fordham University, New York, NY 10023}
\email{hmoon8@fordham.edu}

\author{Luca Schaffler}
\address{Department of Mathematics, KTH Royal Institute of Technology, SE-100 44 Stockholm, Sweden}
\email{lucsch@math.kth.se}

\title[KSBA compactification of moduli of K3 surfaces]{KSBA compactification of the moduli space of K3 surfaces with purely non-symplectic automorphism of order four}

\begin{document}


\maketitle

\begin{abstract}
We describe a compactification by KSBA stable pairs of the five-dimensional moduli space of K3 surfaces with purely non-symplectic automorphism of order four and $U(2)\oplus D_4^{\oplus2}$ lattice polarization. These K3 surfaces can be realized as the minimal resolution of the double cover of $\PP^1\times\PP^1$ branched along a specific $(4,4)$ curve. We show that, up to a finite group action, this stable pair compactification is isomorphic to Kirwan's partial desingularization of the GIT quotient $(\PP^1)^8\git\SL_2$ with the symmetric linearization.
\end{abstract}


\section{Introduction}

Recent advances in algebraic geometry including the minimal model program and the boundedness for stable pairs, enables us to investigate compactifications of moduli spaces of higher dimensional algebraic varieties, in particular the Koll\'ar, Shepherd-Barron, and Alexeev (\emph{KSBA}) moduli space of stable pairs \cite{KSB88, Ale96, Kol18}. However, the geometry of the moduli spaces of higher dimensional varieties is extremely complicated. For instance, these moduli spaces are rarely irreducible, and they may even have arbitrary singularity types \cite{Mne85, Vak06}. 

Nonetheless, sometimes we may understand in detail the geometry of moduli spaces of special algebraic varieties of interest. These explicit moduli spaces are beneficial because in many cases the generalities are out of reach, and also they reveal interesting geometric behaviors (see \cite{AP09, Sch16, GMGZ18, AB19, AET19, DH21}). In this paper, we study one of such explicit examples: the moduli space of K3 surfaces with a purely non-symplectic automorphism of order four (see Definition~\ref{definitionofpurelynonsymplecticautomorphism}). The geometry of K3 surfaces with purely non-symplectic automorphism of order four was studied in \cite{AS15}, and for a survey on the subject in general we refer to \cite{Zha07}.

Here we state our main result. For the precise definitions and terminology, see \S\,\ref{sec:K3eightpoints}. Consider the moduli space of K3 surfaces $\widetilde{X}$ with a purely non-symplectic automorphism of order four together with a $U(2) \oplus D_{4}^{\oplus 2}$ lattice polarization. By Kond\=o's work \cite{Kon07}, there is a five dimensional irreducible moduli space $\bM$ of such K3 surfaces. These K3 surfaces can be obtained by taking the minimal resolution of the double cover $X$ of $\PP^{1} \times \PP^{1}$ branched along an appropriate divisor $B$ of class $(4,4)$. 

We adopt the KSBA theory to compactify $\bM$. Let $\overline{\bK}$ be the normalization of the closure in the KSBA moduli space of stable pairs of the locus parametrizing $(X, \epsilon R)$ where $X$ is the double cover of $\PP^{1} \times \PP^{1}$ branched along $B$, $R$ is the ramification divisor, and $0<\epsilon\ll1$. 

\begin{theorem}\label{thm:mainthm}
The KSBA compactification $\overline{\bK}$ is isomorphic to $\bP/H$ where $H \cong (S_{4} \times S_{4}) \rtimes S_{2}$ and $\bP$ is the partial desingularization of the GIT quotient $(\PP^1)^8\git\SL_2$ with the symmetric linearization. 
\end{theorem}

Note that $\bP$ has another moduli theoretic interpretation: $\bP$ is isomorphic to the Hassett's moduli space of weighted pointed curves $\overline{\rM}_{0,\left(\frac{1}{4}+\epsilon\right)^8}$ \cite[Theorem 1.1]{KM11}.

\subsection{K3 surfaces from eight points on $\PP^{1}$ and GIT}

Here we elaborate more on the construction of these K3 surfaces. Fix eight distinct points $([\lambda_1:1],\ldots,[\lambda_8:1])$ on $\PP^1$. Let $[x_{0}:x_{1}], [y_{0}:y_{1}]$ be the homogeneous coordinates of $\PP^{1} \times \PP^{1}$, and define $B$ to be the following curve of the class $(4, 4)$:
\[
	y_{0}y_{1}\left(y_0^2\prod_{i=1}^4(x_0-\lambda_ix_1)+
	y_1^2\prod_{i=5}^8(x_0-\lambda_ix_1)\right)=0.
\]
Let $X$ be the double cover of $\PP^{1} \times \PP^{1}$ branched along $B$. In \cite{Kon07}, it was shown that 1) the minimal resolution $\widetilde{X}$ of $X$ is a K3 surface with a purely non-symplectic automorphism of order four, so $\bU \subseteq (\PP^{1})^{8}$ parametrizing eight distinct points also parametrizes K3 surfaces, 2) the construction is $\SL_{2}$-invariant and $S_{8}$-invariant (see \cite[\S\,2.1]{Kon07} and \S\,\ref{S8actionandisok3s} of this paper), so $\bM = \bU/\SL_{2}/S_{8}$ can be regarded as a parameter space of such K3 surfaces. The involution $([x_0:x_1],[y_0:y_1])\mapsto([x_0:x_1],[y_0:-y_1])$ is lifted to a purely non-symplectic automorphism $\sigma$ of order four on $\widetilde{X}$.

Let $\rH^{2}(\widetilde{X}, \ZZ)^{+}$ be the invariant subspace of $\rH^{2}(\widetilde{X}, \ZZ)$ with respect to the $(\sigma^2)^*$-action. Then for any $\bp \in \bM$, the associated $\rH^{2}(\widetilde{X}, \ZZ)^{+}$ is a primitive sublattice of $\NS(\widetilde{X})$ isometric to $U(2) \oplus D_{4}^{\oplus 2}$. In summary, the GIT quotient $(\PP^{1})^{8}\git \SL_{2}/S_{8}$ with the symmetric linearization can be thought of as a compactification of the moduli space of K3 surfaces in analysis. 

\subsection{KSBA compactification}

To adopt the KSBA theory in this context, one has to choose an ample divisor $A$ on $\widetilde{X}$ and make a pair $(\widetilde{X}, A)$. However, we make two minor modifications to the moduli problem we consider. First of all, instead of taking an ample divisor $A$, we choose a big and nef divisor on $\widetilde{X}$, which is the pull-back of the ramification divisor $R$ on $X$. This makes the description of the parameter space more accessible by using the theory of abelian covers \cite{AP12}. Secondly, instead of taking the entire linear system of $R$, we just take $R$ to make a five-dimensional moduli space of pairs.

Technically, the resulting moduli space $\overline{\bK}$ is a compactification of a finite cover of $\bM$ because the same K3 surface with a purely non-symplectic automorphism of order four may be regarded as a minimal resolution of the double cover of $\PP^{1} \times \PP^{1}$ in several ways. However, compared to other ways of compactifying $\bM$, the compactification $\overline{\bK}$ has the merit to have an explicit moduli theoretic meaning. 

\subsection{Relation to Hodge theoretic compactifications}

Hodge theory is another standard tool one could use to compactify a given moduli space of varieties. Here we leave some related works. In \cite{DM86}, Deligne and Mostow proved that the symmetric GIT quotient $(\PP^1)^8\git\SL_2$ is isomorphic to the Satake--Baily--Borel compactification $\overline{\cB/\Gamma}^*$ of the quotient of a five-dimensional complex ball $\cB$ by an arithmetic group $\Gamma$. This is proved using periods of a family of curves arising as the $\mathbb{Z}/4\mathbb{Z}$-covers of $\PP^1$ branched along eight points. 

There are two interesting problems. One may wonder how GIT compactifications and Hodge theoretic compactifications are related. In the literature, one celebrated example of interaction between Hodge theory and GIT is the case of K3 surfaces with a degree two polarization. More precisely, in \cite{Loo86} it is shown that a small partial resolution of the Satake--Baily--Borel compactification for such K3 surfaces is isomorphic to a partial Kirwan desingularization of the GIT quotient for sextic plane curves. The moduli space of K3 surfaces which are double covers of $\PP^1\times\PP^1$ branched along a curve of class $(4,4)$ were recently investigated in \cite{LO18} from the point of view of GIT and Hodge theory. In our context of moduli space of K3 surfaces with purely non-symplectic automorphism of order four, the GIT compactification and the Satake--Baily--Borel compactification give the same answer after quotienting by $S_8$ \cite{Kon07}. 

Usually, GIT and Hodge theoretic compactifications do not have a strong modular interpretation. Thus we have the second interesting problem: finding a modular compactification of the given moduli space. From the perspective of moduli theory, the KSBA compactification is arguably the best known theoretical approach. Motivated by this, in \cite{GKS21} it is shown that the KSBA compactification $\overline{\bK}$ is isomorphic to the quotient by $(S_4\times S_4)\rtimes S_2$ of the unique toroidal compactification of the above Deligne--Mostow ball quotient $\mathcal{B}/\Gamma$.

\subsection{Structure of the paper}

The paper is organized as follows. In \S\,\ref{sec:K3eightpoints} we review Kond\=o's construction of the $5$-dimensional family of K3 surfaces with purely non-symplectic automorphism of order four and $U(2) \oplus D_{4}^{\oplus 2}$ lattice polarization. In \S\,\ref{stablepairsandKSBAcompactification} we recall the notion of stable pair, their moduli functor, and the theory of abelian covers. \S\,\ref{partialdesingularizationofGITquotients} contains a brief summary of Kirwan's partial desingularization \cite{Kir85}, which is applied to the case of $(\mathbb{P}^1)^8\git\SL_2$. In \S\,\ref{explicitcalculationofstablereplacement} we study the KSBA limits of specific one-parameter degenerations of stable pairs $(\PP^1\times\PP^1,\frac{1+\epsilon}{2}B)$. These calculations are then used in \S\,\ref{proofofmainthm} to finally prove Theorem~\ref{thm:mainthm}. 

We work over $\mathbb{C}$.


\section*{Aknowledgements}
We would like to express our gratitude to Valery Alexeev, Maksym Fedorchuk, Paul Hacking, Giovanni Inchiostro, Jenia Tevelev, and Alan Thompson for helpful conversations and suggestions. We also thank the anonymous referee for the valuable comments. Most part of this work was done while the first author was a member of the Institute for Advanced Study. The first author was partially supported by the Minerva Research Foundation.


\section{K3 surfaces from eight points on $\PP^1$}\label{sec:K3eightpoints}


\subsection{Kond\=o's construction}
\label{kondo'sconstruction}

\begin{definition}
A normal surface $X$ is called an \emph{ADE K3 surface} if its minimal resolution is a smooth K3 surface, or equivalently, 
\begin{enumerate}
\item $X$ has only ADE singularities (thus it is Gorenstein);
\item $\omega_{X} \cong \cO_{X}$;
\item $\rH^{1}(X, \cO_{X}) = 0$. 
\end{enumerate}
\end{definition}

In \cite{Kon07}, a K3 surface is constructed from the data of eight distinct points on $\PP^1$ as follows. Up to the natural $\SL_2$-action on $\PP^1$, we may assume that the eight points are in the form $[\lambda_1:1],\ldots,[\lambda_8:1]$. Let $C$ be the curve in $\PP^1\times\PP^1$ given by
\begin{equation}\label{eqn:C}
	y_0^2\prod_{i=1}^4(x_0-\lambda_ix_1)+
	y_1^2\prod_{i=5}^8(x_0-\lambda_ix_1)=0,
\end{equation}
where $([x_0:x_1],[y_0:y_1])$ are coordinates in $\PP^1\times\PP^1$. If all $\lambda_{i}$'s are distinct, $C$ is a smooth curve. Let $L_i$ be the line $y_i=0$, $i=1,2$. The double cover $\pi : X \to \PP^1 \times \PP^1$ branched along $C+L_0+L_1$, which has bidegree $(4,4)$, has eight $A_1$ singularities which lie above $C\cap L_0$ and $C\cap L_1$. The minimal resolution $\rho : \widetilde{X} \to X$ of this double cover is a K3 surface. Thus $X$ is an ADE K3 surface with a polarization $F := \pi^{*}\cO(1, 1)$ of degree $4$. Also $\widetilde{X}$ carries a natural big and nef polarization $\rho^{*}F$.

\begin{definition}
\label{definitionofpurelynonsymplecticautomorphism}
An automorphism $\sigma$ of a K3 surface is \emph{non-symplectic} if the induced automorphism on the global sections of the canonical sheaf is not the identity. In addition, we say that $\sigma$ is \emph{purely} non-symplectic if all its non-trivial powers are non-symplectic.
\end{definition}

As Kond\=o described in \cite[\S\,2]{Kon07}, a K3 surface $\widetilde{X}$ as above admits a purely non-symplectic automorphism $\sigma$ of order four given by the lift of the involution 
\[
	([x_{0}:x_{1}], [y_{0}:y_{1}]) \mapsto 
	([x_{0}:x_{1}], [y_{0}:-y_{1}]).
\]
The lattice $\rH^{2}(\widetilde{X}, \ZZ)^{+} := \{x \in \rH^{2}(\widetilde{X}, \ZZ)\;|\; (\sigma^{2})^{*}(x) = x\}$, which embeds primitively into $\NS(\widetilde{X})$, is isometric to $U(2) \oplus D_{4}^{\oplus 2}$, and $\rH^{2}(\widetilde{X}, \ZZ)^{-} := \{x \in \rH^{2}(\widetilde{X}, \ZZ)\;|\; (\sigma^{2})^{*}(x) = -x\}$ is isometric to $U \oplus U(2) \oplus D_{4}^{\oplus 2}$ \cite[Lemma 2.2]{Kon07}. 

\begin{definition}\label{def:moduliM}
Let $\bM$ be the coarse moduli space of K3 surfaces $\widetilde{X}$ with a purely non-symplectic automorphism of order four such that $\rH^{2}(\widetilde{X}, \ZZ)^{+}$ is isometric to $U(2) \oplus D_{4}^{\oplus 2}$. 
\end{definition}

Once $M := \rH^{2}(\widetilde{X}, \ZZ)^{+}$ is identified with $U(2) \oplus D_{4}^{\oplus 2}$, the lattice $N := \rH^{2}(\widetilde{X}, \ZZ)^{-}\cong U\oplus U(2) \oplus D_{4}^{\oplus 2}$ is given because $N = M^{\perp}$. On $N \otimes \CC$, the linear map $\sigma^{*}$ has minimal polynomial $x^2+1$, which implies that $\sigma^*$ is diagonalizable. Moreover, the only possible eigenvalues of $\sigma^*$ are $\pm \sqrt{-1}$. These both occur with the same multiplicity because $\sigma^*$ is a real operator. Hence, if $V$ denotes the eigenspace for $\sqrt{-1}$, then $\dim V = \frac{1}{2}\dim N \otimes \CC = 6$. Thus from \cite[\S\,11]{DK07} (see also \cite[\S\,1]{AS15}), $\bM$ is the quotient of the ball $\{[z]\in\PP(V)\mid z\cdot\overline{z}>0\}$ by an appropriate arithmetic group. In particular, it is a 5-dimensional irreducible analytic variety. 

\begin{remark}
Note that the family of K3 surfaces parametrized by $\bM$ in Definition~\ref{def:moduliM} is the same as the family of K3 surfaces with a purely non-symplectic automorphism of order four and a $U(2)\oplus D_4^{\oplus2}$ lattice polarization. This is true because the very general member $\widetilde{X}$ of the latter family has $\rH^2(\widetilde{X},\mathbb{Z})^{+}=\NS(\widetilde{X})\cong U(2)\oplus D_4^{\oplus2}$.
\end{remark}

The above construction of ADE K3 surfaces can be relativized. Let $([a_1:b_1],\ldots,[a_8:b_8])$ be coordinates in $(\PP^1)^8$. Consider the hypersurface $\cC \subseteq(\PP^1)^8\times \PP^1 \times \PP^1$ given by
\[
	y_0^2\prod_{i=1}^4(b_ix_0-a_ix_1)+
	y_1^2\prod_{i=5}^8(b_ix_0-a_ix_1)=0,
\]
which has multidegree $(4,\ldots,4, 2, 2)$. It can be understood as a family of curves over $(\PP^1)^8$. Let $\bU\subseteq(\PP^1)^8$ be the open subset consisting of $8$-tuples of distinct points. Let $\cX \rightarrow \bU \times \PP^1 \times \PP^1$ be the double cover branched along $(\cC+\cL_0+\cL_1)|_{\bU}$, where $\cL_i := V(y_i) \subseteq \bU \times \PP^1 \times \PP^1$ for $i=0,1$. Since an $\SL_{2}$-orbit in $\bU$ parametrizes isomorphic ADE K3 surfaces, $\bU/\SL_{2}$ is a five dimensional parameter space of ADE K3 surfaces with a purely non-symplectic automorphism of order four. 

Note that there is a natural $S_{8}$ action on $\bU/\SL_{2}$ which permutes the eight points.

\begin{definition}\label{def:H}
Let $H \cong (S_{4} \times S_{4}) \rtimes S_{2}$ be the subgroup of permutations of $S_8$ which is generated by the permutations of the first four points, the permutations of the last four points, and the involution which exchanges the set of first four points and the set of last four points. 
\end{definition}

Any ADE K3 surfaces parametrized by an $H$-orbit are isomorphic to each other because $C + L_{0} + L_{1}$ is $(S_{4} \times S_{4})$-invariant, and the $S_{2}$-action induces an isomorphism of associated surfaces derived by the involution $[y_{0}:y_{1}] \to [y_{1}:y_{0}]$. 

Furthermore, it was shown by Kond\=o that two ADE K3 surfaces as above are isomorphic if and only if the associated points on $\bU/\SL_{2}$ are in the same $S_{8}$-orbit \cite[\S\,3.7]{Kon07}. We will come back to this $S_{8}$-invariance in \S\,\ref{S8actionandisok3s}, where we discuss it when some of the eight points on $\PP^1$ collide. In particular, $\bU/\SL_2/S_8$ is a dense open subset of $\bM$.


\subsection{Degenerate point configurations and GIT}

Consider the diagonal $\SL_2$-action on $(\PP^1)^8$ together with the natural symmetric linearization $\mathcal{O}(1, \ldots, 1)$. A point in $(\PP^1)^8$ is stable (resp. semi-stable) if and only if at most three (resp. four) points coincide. We denote the semi-stable locus (resp. stable locus) by $((\PP^{1})^{8})^{ss}$ (resp. $((\PP^{1})^{8})^{s}$) \cite[\S\,3]{MFK94}. 

\begin{lemma}
The construction in \S\,\ref{kondo'sconstruction} yields an ADE K3 surface for any $\bp\in((\PP^1)^8)^{s}$.  
\end{lemma}

\begin{proof}
Let $\bp = (p_{1},\ldots,p_{8})$ be a stable point configuration with at least one collision. Up to $H$-symmetry, it is clear that Table \ref{tbl:degptconf} describes all the possibilities. By a local computation, we see that the double cover $X$ of $\PP^{1} \times \PP^{1}$ branched along $C + L_{0} + L_{1}$ has only ADE singularities. The two conditions $\omega_{X} \cong \cO_{X}$ and $\rH^{1}(X, \cO_{X}) = 0$ are easy to check using the fact that $X$ is the double cover of $\PP^{1} \times \PP^{1}$ branched along a divisor of class $(4, 4)$. 
\end{proof}

\begin{table}[!ht]
\caption{Degenerate point configurations and singularities on the double cover.}
\label{tbl:degptconf}
\begin{tabular}{|c|c|c|}
\hline
collision & analytic local equation & singularity of double cover\\ \hline \hline
$p_{1} = p_{2} = [0:1]$ & $y_{1}(x_{0}^{2}+y_{1}^{2}) = 0$ & $D_{4}$\\ \hline
$p_{1} = p_{5} = [0:1]$ && $A_{1}$ singularities \\ \hline
$p_{1} = p_{2} = p_{5} = [0:1]$ & $x_{0}y_{1}(x_{0} + y_{1}^{2}) = 0$ & $D_{6}$\\ \hline
$p_{1} = p_{2} = p_{3} = [0:1]$ & $y_{1}(x_{0}^{3}+y_{1}^{2}) = 0$ & $E_{7}$\\ \hline
\end{tabular}
\end{table}

If $\bp$ is strictly semi-stable, then the associated double cover has worse singularities and it is not an ADE K3 surface. Up to the $H$-action, there are three cases:
\begin{enumerate}
\item $p_{1} = p_{2} = p_{5} = p_{6}$;
\item $p_{1} = p_{2} = p_{3} = p_{5}$;
\item $p_{1} = p_{2} = p_{3} = p_{4}$.
\end{enumerate}

For cases (2) and (3), we may compute the minimal resolution of the double cover $X \to \PP^{1} \times \PP^{1}$ branched along $C+L_0+L_1$ using the \emph{canonical resolution} method \cite[III.7]{BHPV04}. The exceptional locus in case (2) is a genus one curve of self-intersection $-2$, and in case (3) is a genus one curve of self intersection $-1$. These singularities are known as $\widetilde{E}_7$ and $\widetilde{E}_8$ respectively (see \cite[\S\,7.6]{Ish18}). In case (1), the ramification divisor is not even reduced. 

For later purpose, the strictly semi-stable points with maximal dimensional stabilizer group are important. There are three types of semi-stable points with positive dimensional stabilizer group, which is isomorphic to $\CC^{*}$. 

\begin{definition}
Let $\bp=(p_1,\ldots,p_8)$ be a strictly semi-stable point configuration. We say that $\bp$ is of \textsf{type a} if, up to $H$-action, it is of the form $p_{1} = p_{2} = p_{5} = p_{6}$ and $p_{3} = p_{4} = p_{7} = p_{8}$. Similarly, we say that $\bp$ is of \textsf{type b} if it is of the form $p_{1} = p_{2} = p_{3} = p_{5}$ and $p_{4} = p_{6} = p_{7} = p_{8}$. Finally, we say that $\bp$ is of \textsf{type c} if it is of the form $p_{1} = p_{2} = p_{3} = p_{4}$ and $p_{5} = p_{6} = p_{7} = p_{8}$. 
\end{definition}

The corresponding curve $C\subseteq \PP^1 \times \PP^1$ is given by:
\begin{itemize}
\item (\textsf{type a}) $(x_0-\lambda_1x_1)^2 (x_0-\lambda_3x_1)^2 (y_0^2+y_1^2)=0$;
\item (\textsf{type b}) $(x_0-\lambda_1x_1)(x_0-\lambda_4x_1) (y_0^2(x_0-\lambda_1x_1)^2 + y_1^2(x_0-\lambda_4x_1)^2)=0$;
\item (\textsf{type c}) $y_0^2(x_0-\lambda_1x_1)^4 + y_1^2(x_0-\lambda_5x_1)^4=0$.
\end{itemize}

Each connected component of such strictly semi-stable points is isomorphic to $\PP^{1} \times \PP^{1} \setminus \Delta$, where $\Delta$ is the diagonal. The locus of \textsf{type a} point configurations has 18 connected components, that of \textsf{type b} configurations has 16 connected components, and that of \textsf{type c} configurations has a unique component. 

\begin{remark}
The associated double covers of $\PP^1\times\PP^1$ branched along $C+L_0+L_1$ with $C$ of \textsf{type a}, \textsf{type b}, and \textsf{type c} appear in Shah's list \cite[Theorem 4.8, B, Type II, (i)--(iii)]{Sha81}.
\end{remark}


\subsection{$S_8$-invariance of Kond\=o's K3 surfaces}
\label{S8actionandisok3s}

Let $X$ be an ADE K3 surface associated to $8$ distinct points on $\PP^1$. In \cite{Kon07} it is observed that the isomorphism class of $X$ is independent from the ordering of the eight points. Here we give a proof which is valid for all stable configurations. 

\begin{proposition}
Let $X$ be an ADE K3 surface which is the double cover of $\PP^1\times\PP^1$ branched along the curve
\[
	B : y_0y_1\left(y_0^2\prod_{i=1}^4(x_0-\lambda_ix_1)
	+y_1^2\prod_{i=5}^8(x_0-\lambda_ix_1)\right)=0.
\]
Let $\tau\in S_8$ and define $X_\tau$ to be the ADE K3 surface obtained as the double cover of $\PP^1\times\PP^1$ branched along the curve
\[
	B_\tau : y_0y_1\left(y_0^2\prod_{i=1}^4(x_0-\lambda_{\tau(i)}x_1)
	+y_1^2\prod_{i=5}^8(x_0-\lambda_{\tau(i)}x_1)\right)=0.
\]
Then $X \cong X_\tau$.
\end{proposition}

\begin{proof}
It is enough to prove the case where $\tau$ is a transposition $(ij)$. Because of the symmetry of $B$, we may assume that $(ij) = (15)$. Since smooth K3 surfaces have trivial canonical class, it is sufficient to show that $X$ and $X_{(15)}$ are birational. 

Over the affine patch $x_1y_1\neq 0$, $X$ is defined in $\mathbb{A}_{(x,y,z)}^3$ by
\[
	z^2=y\left(y^2\prod_{i=1}^4(x-\lambda_i)
	+\prod_{i=5}^8(x-\lambda_i)\right),
\]
where we set $y=y_0/y_1$ and $x=x_0/x_1$. Consider the birational transformation
\begin{eqnarray*}
	\mathbb{A}_{(\xi,\eta,\zeta)}^3& \dashrightarrow
	& \mathbb{A}_{(x,y,z)}^3\\
	(\xi,\eta,\zeta)&\mapsto& 
	\left(\xi,\frac{\xi-\lambda_5}{\xi-\lambda_1}\eta,
	\frac{\xi-\lambda_5}{\xi-\lambda_1}\zeta\right).
\end{eqnarray*}
Under this birational transformation, the pull-back of $X$ satisfies
\[
	\frac{(\xi-\lambda_5)^2}{(\xi-\lambda_1)^2}\zeta^2
	=\frac{\xi-\lambda_5}{\xi-\lambda_1}\eta
	\left(\frac{(\xi-\lambda_5)^2}{(\xi-\lambda_1)^2}\eta^2
	\prod_{i=1}^4(\xi-\lambda_i)+\prod_{i=5}^8(\xi-\lambda_i)\right)
\]
\[
	\implies\zeta^2=\eta\left(\frac{\xi-\lambda_5}{\xi-\lambda_1}\eta^2
	\prod_{i=1}^4(\xi-\lambda_i)
	+\frac{\xi-\lambda_1}{\xi-\lambda_5}\prod_{i=5}^8(\xi-\lambda_i)\right),
\]
which is the equation for $X_{(15)}$ over the affine patch $x_1y_1\neq0$. Thus they are birational.
\end{proof}

\section{Stable pairs and KSBA compactification}
\label{stablepairsandKSBAcompactification}


In this section we recall the definition of stable pair, their moduli spaces, and the theory of abelian covers. Our main references are \cite{Ale15, AP12, Kol13, Kol18}.

\subsection{Definition of stable pair}

\begin{definition}
Let $X$ be a variety and let $D$ be a $\mathbb{Q}$-divisor on $X$ with coefficients in $(0,1]$. A pair $(X,D)$ is \emph{semi-log canonical} if:
\begin{enumerate}
\item $X$ is demi-normal (that is, $X$ is $S_2$ and its codimension $1$ points are either regular or ordinary nodes);
\item If $\nu : X^\nu\rightarrow X$ is the normalization with conductors $E\subseteq X$ and $E^\nu\subseteq X^\nu$, then the support of $E$ does not contain any irreducible component of $D$;
\item $K_X+D$ is $\mathbb{Q}$-Cartier;
\item The pair $(X^\nu,E^\nu+\nu_*^{-1}D)$ is log canonical. More precisely, for each connected component $Z$ of $X^\nu$ the pair $(Z,(E^\nu+\nu_*^{-1}D)|_Z)$ is log canonical, where $\nu_*^{-1}D$ denotes the strict transform of $D$. For the definition of log canonical we refer to \cite[Definition 2.8]{Kol13}.
\end{enumerate}
\end{definition}

\begin{definition}
A pair $(X,D)$ is \emph{stable} if the following conditions are satisfied:
\begin{enumerate}
\item $(X,D)$ is a semi-log canonical pair;
\item $K_X+D$ is ample.
\end{enumerate}
\end{definition}

Let $\epsilon$ be a sufficiently small positive rational number. 

\begin{lemma}
Let $([\lambda_1:1],\ldots,[\lambda_8:1])\in(\PP^1)^8$ be a stable point and let $C\subseteq\PP^1\times\PP^1$ as in Equation \eqref{eqn:C}. Let $B:=C+L_0+L_1$. Then
\begin{enumerate}
\item $\left(\PP^1\times\PP^1,\frac{1+\epsilon}{2}B\right)$ is a stable pair;
\item $(K_{\PP^1\times\PP^1}+\frac{1+\epsilon}{2}B)^2=8\epsilon^2$.
\end{enumerate}
In short, $((\PP^{1})^{8})^{s}$ parametrizes stable pairs $(\PP^{1} \times \PP^{1}, \frac{1+\epsilon}{2}B)$. 
\end{lemma}

\begin{proof}
If the eight points are distinct, then the divisor $B$ is simple normal crossing, hence $\left(\PP^1\times\PP^1,\frac{1+\epsilon}{2}B\right)$ is semi-log canonical. If some of the eight points coincide, then the semi-log canonicity of the pair follows by inspecting all cases in Table \ref{tbl:degptconf}. The ampleness of $K_{\PP^1\times\PP^1}+\frac{1+\epsilon}{2}B$ and the equality in (2) follow from $K_{\PP^1\times\PP^1}+\frac{1+\epsilon}{2}B\sim2\epsilon(1,1)$. 
\end{proof}


\subsection{Stable pairs and finite abelian covers}
\label{coveringtrick}

For the reader's convenience, we recall the following well known facts about stable pairs and finite abelian covers. For a reference, see \cite{AP12}.

\begin{definition}
A \emph{morphism of pairs} $f : (X, B_{X}) \rightarrow (P, B_{P})$ is a morphism $f : X \rightarrow P$ mapping $\mathrm{Supp}(B_{X})$ to $\mathrm{Supp}(B_{P})$. If $G$ is a finite abelian group, then a morphism $\pi : (X,B_{X})\rightarrow (P,B_{P})$ is called a \emph{$G$-cover} if:
\begin{itemize}
\item $\pi : X\rightarrow P$ is the quotient morphism for a generically faithful action of $G$;
\item $\pi$ is branched along $\mathrm{Supp}(B_{P})$ and ramified at $\mathrm{Supp}(B_{X})$;
\item $K_X+B_{X} = \pi^*(K_P+B_{P})$.
\end{itemize}
\end{definition}

\begin{lemma}[\protect{\cite[Lemma 2.3]{AP12}}]
Let $\pi : (X, B_{X}) \to (P, B_{P})$ be a $G$-cover for a finite abelian group $G$. Then $(X, B_{X})$ is stable if and only if $(P, B_{P})$ is stable.
\end{lemma}

\begin{remark}
More precisely, \cite[Lemma 2.3]{AP12} guarantees that $(X, B_{X})$ is semi-log canonical if and only if $(P, B_{P})$ is semi-log canonical. We have that $K_X+B_{X}$ is ample if and only if $\pi^*(K_P+B_{P})$ is because $\pi$ is a finite covering (see \cite[Proposition 1.2.13 and Corollary 1.2.28]{Laz04}).
\end{remark}

\begin{example}
\label{exampleofcoverofinterest}
Let $P := \PP^{1} \times \PP^{1}$ and let $B := C + L_{0} + L_{1}$ as in \S\,\ref{kondo'sconstruction}. Let $X$ be the double cover of $P$ branched along $B$ and let $R$ be the ramification divisor. Then it is straightforward to check that 
\[
	\pi : (X, \epsilon R) \to \left(P, \frac{1+\epsilon}{2}B\right)
\]
is a $(\ZZ/2\ZZ)$-cover. In particular, $(K_X+\epsilon R)^2=\pi^*\left(K_P+\frac{1+\epsilon}{2}B\right)^2=16\epsilon^2$.
\end{example}


\subsection{The moduli functor}
\label{viehwegmodulifunctor}

\begin{definition}
We fix constants $d,N\in\mathbb{Z}_{>0}$, $C\in\mathbb{Q}_{>0}$, and $\underline{b}=(b_1,\ldots,b_n)$ with $b_i\in(0,1]\cap\mathbb{Q}$ and $Nb_i\in\mathbb{Z}$ for all $i=1,\ldots, n$. The \emph{Viehweg's moduli stack} $\overline{\cV} := \overline{\cV}_{d, N, C, \underline{b}}$ is defined as follows. For any reduced $\CC$-scheme $S$, $\overline{\cV}_{d,N,C,\underline{b}}(S)$ is the set of proper flat families $\cX \rightarrow S$ together with a divisor $\cB=\sum_i b_i \cB_i$ satisfying:
\begin{enumerate}
\item For all $i=1,\ldots, n$, $\cB_i$ is a codimension one closed subscheme such that $\cB_{i} \to S$ is flat at the generic points of $\cX_{s} \cap \mathrm{Supp}(\cB_{i})$ for every $s \in S$;
\item Every geometric fiber $(X,B)$ is a stable pair of dimension $d$ with $(K_X+B)^d=C$;
\item There exists an invertible sheaf $\cL$ on $\cX$ such that for every geometric fiber $(X,B)$ one has $\cL|_X\cong\mathcal{O}_X(N(K_X+B))$.
\end{enumerate}
In what follows, we refer to such families $(\mathcal{X},\mathcal{B})\rightarrow S$ as \emph{families of stable pairs}.
\end{definition}

\begin{remark}\label{rem:DMstack}
Over characteristic zero, an algebraic stack is Deligne--Mumford if and only if all parametrized objects have finite automorphism groups \cite[Remark 8.3.4]{Ols16}. Since every stable pair parametrized by $\overline{\cV}$ has a finite automorphism group \cite[Theorem 1.20]{Fuj14}, $\overline{\cV}$ is a Deligne--Mumford stack.
\end{remark}

Let $\underline{b}$ be very general \cite[\S\,1.5.3]{Ale15}. For a suitably chosen positive integer $N$ depending on $d,C$, and $\underline{b}$ (which does not need to be specified, see \cite[\S\,3.13]{Ale96}), the stack $\overline{\cV}$ above is coarsely represented by a projective scheme by \cite[Theorem 1.6.1]{Ale15}.

We now describe our special case of interest.

\begin{definition}
\label{def:KSBAcomp}
Let $\overline{\cV}$ be the Viehweg's moduli stack for $d = 2$, $C = 16\epsilon^{2}$ for a small general $\epsilon > 0$, and $\underline{b} = (b_{1}) = (\epsilon)$. The pairs $(X, \epsilon R)$ where $X$ is an ADE K3 surface in \S\,\ref{kondo'sconstruction} and $R = \frac{1}{2}\pi^{*}B$ is the ramification divisor, are parametrized by $\overline{\cV}$. Now consider the family of stable pairs
\[
	\left(\cY := \bU \times \PP^1 \times \PP^1, 
	\frac{1+\epsilon}{2}(\cC+\cL_0+\cL_1)|_\bU \right)
	\rightarrow\bU.
\]
(For the definitions of $\bU,\cC,\cL_0,\cL_1$ we refer to \S\,\ref{kondo'sconstruction}.) Let $\cX$ be the double cover of $\cY$ branched along $(\cC + \cL_{0} + \cL_{1})|_{\bU}$ and let $\cR$ be the ramification divisor. Then we obtain a family of stable pairs $(\cX, \epsilon \cR) \to \bU$ where the fibers $X$ are ADE K3 surfaces with a purely non-symplectic automorphism of order four and $R$ is the ramification divisor. Thus we obtain a morphism $\bU \to \overline{\cV}$ and denote by $\overline{\cK}'$ the closure of its image in $\overline{\cV}$. Let $\overline{\bK}'$ be the coarse moduli space corresponding to $\overline{\cK}'$, and denote by $\overline{\bK}$ its normalization. Observe that $\overline{\bK}$ is compactifying $\bU/\SL_2/H$, and we call it the \emph{KSBA compactification} of the moduli space of ADE K3 pairs with purely non-symplectic automorphism of order four and $U(2) \oplus D_{4}^{\oplus 2}$ lattice polarization. As we already pointed out in the introduction, $\overline{\bK}$ is compactifying a finite cover of $\bM$ in Definition~\ref{def:moduliM}, and more precisely a $(S_8/H)$-cover.
\end{definition}

Our ultimate goal is to study the geometry of the compactified moduli space $\overline{\bK}$. To do so, we consider the following other projective moduli space.

\begin{definition}
\label{stackforbasestablepairs}
Let $\overline{\cT}$ be the Viehweg's moduli stack for $d=2$, $C=8\epsilon^2$, $\underline{b}=(b_{1}) = \left(\frac{1+\epsilon}{2}\right)$. Consider again the family of stable pairs $(\cY,\frac{1+\epsilon}{2}(\cC+\cL_0+\cL_1)|_{\bU})\rightarrow\bU$. There is an induced morphism $\bU\rightarrow\overline{\cT}$ and denote by $\overline{\cJ}'$ the closure of its image in $\overline{\cT}$. Let $\overline{\bJ}'$ be the coarse moduli space corresponding to $\overline{\cJ}'$, and denote by $\overline{\bJ}$ its normalization. We have that $\overline{\bJ}$ is compactifying $\bU/\SL_2/H$, and hence it is birational to $\overline{\bK}$.
\end{definition}

In Proposition \ref{compactificationsarethesame} we show that $\overline{\bK} \cong \overline{\bJ}$. (More precisely, we will prove that there is a bijective morphism $\overline{\bK}' \to \overline{\bJ}'$. This does not imply that $\overline{\bK}' \cong \overline{\bJ}'$, but we can conclude that their normalizations are isomorphic.) Thus it is enough to investigate the geometry of $\overline{\bJ}$. 

\begin{remark}

\

\begin{enumerate}
\item Note that, however, the associated stacks $\overline{\cK}'$ and $\overline{\cJ}'$ are not isomorphic because of their stacky structure. Recall that the former parametrizes stable pairs $(X, \epsilon R)$ where $X$ is an ADE K3 surface with a purely non-symplectic automorphism of order four, and the latter stable pairs $(\PP^{1} \times \PP^{1}, \frac{1+\epsilon}{2}B)$. Therefore, a general object parametrized by $\overline{\cK}'$ has an extra $\ZZ/2\ZZ$-action which comes from its double covering structure. 
\item The proof of Proposition \ref{compactificationsarethesame} tells us that there is a bijective morphism $\overline{\bK}' \to \overline{\bJ}'$ between the coarse moduli spaces.
\end{enumerate}
\end{remark}

Recall that $\bU\subseteq(\mathbb{P}^1)^8$ is the locus of eight distinct points. In Definitions~\ref{def:KSBAcomp} and \ref{stackforbasestablepairs} we introduced two compactifications of $\bU/\SL_2/H$, which we denoted by $\overline{\bK}$ and $\overline{\bJ}$. The remaining part of this section is devoted to proving the following.

\begin{proposition}
\label{compactificationsarethesame}
The compactifications $\overline{\bK}$ and $\overline{\bJ}$ are isomorphic.
\end{proposition}

To prove the claim above we need a lemma.

\begin{lemma}
\label{extendactionfamiliesstablepairs}
Let $S$ be a scheme and let $(\cX,\cB_{\cX})\rightarrow S$ be a family of stable pairs. Let $G$ be a finite abelian group and assume there exists a dense open subset $U\subseteq S$ such that $(\cX,\cB_{\cX})|_U$ has a fiberwise generically faithful action of $G$ which ramifies at $\mathrm{Supp}(\cB_{\cX}|_U)$. Then the $G$-action extends to the whole $(\cX,\cB_{\cX})$ giving a fiberwise generically faithful action which ramifies at $\mathrm{Supp}(\cB_{\cX})$.
\end{lemma}

\begin{proof}
Let $g\in G$ be arbitrary. We show that the corresponding action $\alpha_g : \cX|_U\rightarrow\cX|_U$ extends to $\cX$. Consider a resolution of indeterminacies

\begin{center}
\begin{tikzpicture}[>=angle 90]
\matrix(a)[matrix of math nodes,
row sep=2em, column sep=2em,
text height=1.5ex, text depth=0.25ex]
{&\cX'&\\
\cX&&\cX.\\};
\path[->] (a-1-2) edge node[above]{}(a-2-1);
\path[->] (a-1-2) edge node[above right]{$\alpha_g'$}(a-2-3);
\path[dashed,->] (a-2-1) edge node[below]{$\alpha_g$}(a-2-3);
\end{tikzpicture}
\end{center}
Let $\cB_{\cX}'$ to be the strict transform of $\cB_{\cX}$ under $\cX'\rightarrow\cX$. Then $\alpha_g'$ induces a morphism $\alpha_g^\textrm{lc}$ from the log canonical model of $(\cX',\cB_{\cX}')$ to $\cX$ \cite[Definition 1.19]{Kol13}. Since the log canonical model of $(\cX',\cB_{\cX}')$ is $(\cX,\cB_{\cX})$, it follows that $\alpha_g^\textrm{lc}$ is the desired extension of $\alpha_g$.
\end{proof}

\begin{proof}[Proof of Proposition~\ref{compactificationsarethesame}]
Since $\overline{\cK}'$ is a Deligne--Mumford stack (see Remark~\ref{rem:DMstack}), there exists a scheme $A$ and a surjective \'etale morphism $\alpha : A\rightarrow\overline{\cK}'$. Therefore, there exists a family $\cA\rightarrow A$ (we omit the datum of the divisor for simplicity of notation), and every object parametrized by $\overline{\cK}'$ appears as a fiber of $\cA\rightarrow A$ because $\alpha$ is surjective. In particular, there exists a dense open subset $U\subseteq A$ such that $\cA|_U\rightarrow U$ admits a fiberwise $(\mathbb{Z}/2\mathbb{Z})$-action. So we can apply Lemma~\ref{extendactionfamiliesstablepairs} to extend the $(\mathbb{Z}/2\mathbb{Z})$-action to the whole family $\cA\rightarrow A$.

Now, let $S$ be a normal scheme and let $\cX\rightarrow S$ be a family for the stack $\overline{\cK}'$. Consider the algebraic space $S':=S\times_{\overline{\cK}'}A$, which we may assume is a scheme by replacing it by its atlas, if necessary. We have that $S'$ comes with a family $\cX'\rightarrow S'$ obtained by pulling-back $\cA\rightarrow A$ along the morphism $S'\rightarrow A$. Observe that $\cX'\rightarrow S'$ equals the pull-back of $\cX\rightarrow S$ along the morphism $S'\rightarrow S$ because the two compositions $S'\rightarrow A\rightarrow\overline{\cK}'$ and $S'\rightarrow S\rightarrow\overline{\cK}'$ are equal. We can summarize these considerations in the following commutative diagrams of cartesian squares:

\begin{center}
\begin{tikzpicture}[>=angle 90]
\matrix(a)[matrix of math nodes,
row sep=.5em, column sep=2em,
text height=1.5ex, text depth=0.25ex]
{&\cX'&\\
\cX&&\cA\\
&S'&\\
S&&A\\
&\overline{\cK}'.&\\};
\path[->] (a-1-2) edge node[above]{}(a-2-1);
\path[->] (a-1-2) edge node[above]{}(a-2-3);
\path[->] (a-1-2) edge node[above]{}(a-3-2);
\path[->] (a-2-1) edge node[above]{}(a-4-1);
\path[->] (a-2-3) edge node[above]{}(a-4-3);
\path[->] (a-3-2) edge node[above]{}(a-4-1);
\path[->] (a-3-2) edge node[above]{}(a-4-3);
\path[->] (a-4-1) edge node[above]{}(a-5-2);
\path[->] (a-4-3) edge node[above]{}(a-5-2);
\end{tikzpicture}
\end{center}

Since $\cX'\rightarrow S'$ has a fiberwise $(\mathbb{Z}/2\mathbb{Z})$-action, also $\cX\rightarrow S$ does. The quotient of $\cX\rightarrow S$ by this $(\mathbb{Z}/2\mathbb{Z})$-action gives an object for the stack $\overline{\cJ}'$. This guarantees the existence of a morphism of stacks $\overline{\cK}'\rightarrow\overline{\cJ}'$ which induces a morphism $f:\overline{\bK}'\rightarrow\overline{\bJ}'$ of the underlying coarse moduli spaces.

The morphism $f$ is surjective because it is the identity on $\bU/\SL_2/H$. On the other hand, assume $f(p)=q$ and let $(X,\epsilon R)$ (resp. $(\PP^{1} \times \PP^{1}, \frac{1+\epsilon}{2}B)$) be the stable pair parametrized by $p$ (resp. $q$). Then $(X,\epsilon R)$ is the double cover of $(\PP^{1} \times \PP^{1}, \frac{1+\epsilon}{2}B)$, and since such double cover is unique, we have that $f$ is injective as well. Thus $f$ is bijective. 

Now consider the associated morphism $\overline{\bK} \to \overline{\bJ}$ between normalizations. Since it is a quasi-finite surjective birational morphism between normal varieties, by Zariski's Main Theorem it is an isomorphism.
\end{proof}


\section{Partial desingularization of GIT quotients}
\label{partialdesingularizationofGITquotients}


\subsection{Review on partial desingularization}

In this section, we recall the partial desingularization of GIT quotient developed by Kirwan in \cite{Kir85}. 

Let $G$ be a reductive group and let $X$ be a smooth projective variety equipped with an $L$-linearized $G$-action. Suppose that the stable locus $X^{s}(L)$ is nonempty. If $X$ has strictly semi-stable points, so that $X^{s}(L) \subsetneq X^{ss}(L)$, then the GIT quotient $X\git_{L}G$ may have non-finite quotient singularities. Such a singularity can be partially $G$-equivariantly resolved by using Kirwan's partial desingularization. The outcome of the partial desingularization process is a new algebraic variety $X'$ equipped with an $L'$-linearized $G$-action such that: 
\begin{enumerate}
\item There is a dominant projective birational map $X'\git_{L'}G \to X \git_{L}G$;
\item $X' \git_{L'}G$ has finite quotient singularities only.
\end{enumerate}

Assume that $X^{ss}(L) \ne X^{s}(L)$. Let $Y$ be the closed $G$-invariant subvariety of $X^{ss}(L)$ with a maximal dimensional stabilizer group. Let $\widetilde{X}$ be the blow-up of $X^{ss}(L)$ along $Y$ and let $\pi : \widetilde{X} \to X^{ss}(L)$ be the blow-up morphism. Let $E$ be the exceptional divisor. $\widetilde{X}$ is a smooth quasi-projective variety since $Y$ is smooth. For a small $\epsilon > 0$, $L_{\epsilon} := \pi^{*}L \otimes \cO(-\epsilon E)$ is an ample line bundle on $\widetilde{X}$ and the $G$-action is naturally extended to $\widetilde{X}$ and to $L_{\epsilon}$.

The stable and semi-stable loci of $X$ and $\widetilde{X}$ are related as follows:
\[
	\pi^{-1}X^{s}(L) \subseteq \widetilde{X}^{s}(L_{\epsilon})
	\subseteq \widetilde{X}^{ss}(L_{\epsilon}) \subseteq
	\pi^{-1}X^{ss}(L).
\]
Thus there is a natural $G$-equivariant morphism $\widetilde{X}^{ss}(L_{\epsilon}) \to X^{ss}(L)$ which induces a morphism $\overline{\pi} : \widetilde{X}\git_{L_{\epsilon}}G \to X \git_{L}G$ between quotients. On $\widetilde{X}$, a point $x \in \widetilde{X}$ is unstable if the orbit of $\pi(x)$ is not closed in $X^{ss}(L)$ \cite[Lemma 6.6]{Kir85}. Furthermore, the maximal dimension of the stabilizer group strictly decreases. After replacing $X$ by $\widetilde{X}^{ss}$ and $L$ by $L_{\epsilon}$, we can continue this process and it terminates. The result of this procedure is $X'$ and $L'$. 


\subsection{Partial resolution of the parameter space of eight points on the projective line}\label{ssec:P18}

Let $\bX_{0} := (\PP^{1})^{8}$ equipped with a diagonal $\SL_{2}$-action. Let $L = \cO(1, \ldots, 1)$ be the symmetric linearization. We explicitly describe the partial desingularization of $\bX_{0}\git_{L}\SL_{2}$. For the detail, see \cite[\S\,9]{Kir85}. 

Recall that for a point configuration $\bp \in \bX_{0}$, $\bp \in \bX_{0}^{ss}$ if and only if at most four points collide and $\bp \in \bX_{0}^{s}$ if and only if at most three points collide.

If ${[8] \choose 4}$ denotes the set of $4$-element subsets of $[8]:=\{1,\ldots,8\}$, for any $I \in {[8] \choose 4}$, let $\Delta_{I} := \{\bp = (p_{i}) \in \bX_{0}^{ss} \;|\; p_{i} = p_{j} \mbox{ for all } i, j \in I\}$. It is a five dimensional smooth subvariety of $\bX_{0}^{ss}$. The set of strictly semi-stable points is
\[
	\bX_{0}^{ss} \setminus \bX_{0}^{s} = 
	\bigcup_{I \in {[8] \choose 4}}\Delta_{I}.
\]
For $I \in {[8] \choose 4}$, let $\Delta_{I, I^{c}} := \Delta_{I} \cap \Delta_{I^{c}}$. Note that $\Delta_{I, I^{c}} \cong (\PP^{1})^{2}\setminus \PP^{1}$, where the last $\PP^{1}$ is the diagonal. For any $J \ne I, I^{c}$, we have that $\Delta_{I, I^{c}} \cap \Delta_{J, J^{c}} = \emptyset$. So if we let $\Delta_{4, 4} = \bigcup_{I \in {[8] \choose 4}}\Delta_{I, I^{c}}$, then $\Delta_{4,4}$ is a disjoint union of 35 copies of two dimensional smooth closed subvarieties of $\bX_{0}^{ss}$. $\Delta_{4, 4}$ is precisely the set of semi-stable points with positive dimensional stabilizer, which is isomorphic to $\CC^{*}$. 

Let $\bX_{1}$ be the blow-up of $\bX_{0}^{ss}$ along $\Delta_{4, 4}$. Let $E_{I, I^{c}}$ be the exceptional divisor over $\Delta_{I, I^{c}}$. For each $\Delta_{I}$, let $\widetilde{\Delta}_{I}$ be the proper transform of $\Delta_{I}$. Then $\bX_{1}$ is a smooth variety equipped with a linearized $\SL_{2}$-action. Let $\rho : \bX_{1} \to \bX_{0}^{ss}$ be the blow-up morphism. By the recipe of the partial desingularization, it is straightforward to check the following:
\begin{enumerate}
\item If $\rho(\bp) \in \bX_{0}^{s}$, then $\bp$ is stable;
\item If $\bp \in E_{I, I^{c}} \setminus (\widetilde{\Delta}_{I} \cup \widetilde{\Delta}_{I^{c}})$, then $\bp$ is stable;
\item If $\bp \in \widetilde{\Delta}_{I}$, then $\bp$ is unstable. 
\end{enumerate}

In particular, $\bX_{1}^{ss} = \bX_{1}^{s}$. One may check that for every $\bp \in \bX_{1}^{s}$, the stabilizer group is isomorphic to $\{\pm 1\}$. Therefore, the GIT quotient $\bX_{1}^{s}/\SL_{2}$ is a smooth variety. 

\begin{definition}\label{def:P}
Let $\bP$ be the partial desingularization of $(\PP^{1})^{8}\git \SL_{2}$ with symmetric linearization, that is, $\bP := \bX_{1}^{s}/\SL_{2}$.
\end{definition}

Note that both $\bX_{0}$ and $\bX_{1}$ have $\bU$ as an open subset. So there is an open embedding $\bU/\SL_{2} \hookrightarrow \bP = \bX_{1}^{s}/\SL_{2}$.

\begin{remark}\label{rmk:modulimeaningofP}

\

\begin{enumerate}
\item By \cite[Theorem 1.1]{KM11}, $\bP$ is isomorphic to $\overline{\rM}_{0, \left(\frac{1}{4}+\epsilon\right)^{8}}$, the moduli space of stable rational curves with eight marked points of weight $\frac{1}{4}+\epsilon$ (see \cite{Has03}).
\item An irreducible component $\overline{E}_{I, I^{c}}$ of the exceptional divisor of $\bar{\pi} : \bP \to (\PP^{1})^{8}\git \SL_{2}$ is isomorphic to $\PP^{2} \times \PP^{2}$. Indeed, because $\SL_{2}$ acts on $\Delta_{I,I^c}$ transitively with stabilizer group $\CC^{*}$, the exceptional set is isomorphic to $\pi^{-1}(\Delta_{I,I^c})\git \SL_{2} \cong \PP^{5}\git \CC^{*}$. Here $\PP^{5} \cong \PP(\cN_{\Delta_{I,I^c}/\bX_{0}^{ss}}|_x)$ for some $x \in \Delta_{I,I^c}$. One can check that the weight decomposition of $\cN_{\Delta_{I,I^c}/\bX_{0}^{ss}}|_x$ is $(2)^{3}\oplus (-2)^{3}$. Thus the GIT quotient $\PP^{5}\git \CC^{*}$ is isomorphic to $\PP^{2} \times \PP^{2}$. 

Alternatively, we may adopt the moduli theoretic meaning (again, \cite[Theorem 1.1]{KM11}) to obtain the same result: $\overline{E}_{I, I^{c}} \cong \overline{\rM}_{0, \left(1, (\frac{1}{4}+\epsilon)^{4}\right)} \times \overline{\rM}_{0, \left(1, (\frac{1}{4}+\epsilon)^{4}\right)} \cong \PP^{2} \times \PP^{2}$. 
\end{enumerate}
\end{remark}


\section{Explicit calculations of stable replacements}
\label{explicitcalculationofstablereplacement}

The first step toward the proof of Theorem \ref{thm:mainthm} is the construction of an extension $(\widetilde{\cY}, \frac{1+\epsilon}{2}\widetilde{\cB}) \to \bX_{1}^{s}$ of the family of pairs $(\cY, \frac{1+\epsilon}{2}\cB) \to \bU$, where recall $\cY=\mathbf{U}\times(\PP^1\times\PP^1)$ and we set $\mathcal{B}:=(\mathcal{C}+\mathcal{L}_0+\mathcal{L}_1)|_{\bU}$ (see Definition~\ref{def:KSBAcomp}). This would induce a functorial morphism $\bX_{1}^{s} \to \overline{\bK}$. The above extension is obtained by modifying $(\cY_1,\frac{1+\epsilon}{2}\cB_1)\to\bX_{1}^{s}$, where $\cY_1:=\bX_{1}^{s}\times(\PP^1\times\PP^1)$ and $\cB_1$ is obtained by pulling-back $\mathcal{C}+\mathcal{L}_0+\mathcal{L}_1$ under the appropriate morphism. We postpone the global analysis of this modification to \S\,\ref{proofofmainthm}. In this section, we describe how the modification of the family goes with concrete examples of one-parameter degenerations. 

Throughout the whole section we adopt the following notation. For a stable pair $(X,B)$, if $\nu : \amalg_iX_i\rightarrow X$ is the normalization map, then we denote by $D_{X_i}$ the conductor divisor on $X_i$. We use the letter $P$ to denote $\PP^1\times\PP^1$, and $\Delta$ to denote the germ of a curve whose uniformizing parameter is $t$. On $P\times\Delta$, we let $\cL_0,\cL_1$ denote the divisors $V(y_0),V(y_1)$ respectively. For simplicity of notation, we denote the fibers $(\cL_0)_0,(\cL_1)_0$ over $t=0$ by $L_0,L_1$ respectively. $\cB$ denotes the divisor $\cC+\cL_0+\cL_1$, where $\cC$ is appropriately defined in each one of the examples that follow, and it depends on the choice of eight general complex numbers $\vec{\lambda} := (\lambda_{1}, \ldots, \lambda_{8}) \in \CC^{8} \subseteq (\PP^{1})^{8}$.


\subsection{Semi-stable points of \protect{\textsf{type a}}}
\label{stablereplacementtypea}

\begin{example}
\label{typeafirstexample}
Consider the one-parameter family of divisors $\cC$ on $P \times \Delta$ given by
\begin{equation*}
\begin{split}
	y_0^2(x_{0} - t\lambda_{1}x_1)(x_{0} - t\lambda_{2}x_1)&
	(x_{0} - \lambda_{3}x_1)(x_{0} - \lambda_{4}x_1) \\
	+ \;
	y_{1}^{2}&(x_{0} - t\lambda_{5}x_1)(x_{0} - t\lambda_{6}x_1)
	(x_{0} - \lambda_{7}x_1)(x_{0} -\lambda_{8}x_1) = 0,
\end{split}
\end{equation*}
which is associated to $\vec{\lambda}_{t} = (t\lambda_{1}, t\lambda_{2}, \lambda_{3}, \lambda_{4}, t\lambda_{5}, t\lambda_{6}, \lambda_{7}, \lambda_{8}) \in (\PP^{1})^{8}$. We have a family of pairs $(P \times \Delta, \frac{1+\epsilon}{2}\cB) \to \Delta$. When $t = 0$, $\vec{\lambda}_{0}$ is strictly semi-stable. For a general $t \ne 0$, the fiber $(P, \frac{1+\epsilon}{2}\cB_{t})$ is a stable pair, thus its double cover branched along $\cB_{t}$ is also a stable ADE K3 pair. On the other hand, note that $\left(P,\frac{1+\epsilon}{2}\cB_0\right)$ is not a stable pair because $\cB_{0} = 2L + C'+L_0+L_1$ has a double line where $L$ is the line $V(x_0)$. Here $C'$ is the $(2, 2)$ divisor given by
\[
	y_0^2(x_{0} - \lambda_{3}x_1)(x_{0} - \lambda_{4}x_1) + 
	y_{1}^{2}(x_{0} - \lambda_{7}x_1)(x_{0} -\lambda_{8}x_1) = 0.
\]
The intersection $L \cap C'$ consists of the two distinct solutions of $y_0^2\lambda_{3}\lambda_{4}+y_{1}^{2}\lambda_{7}\lambda_{8}= 0$ (recall we assumed that the $\lambda_i$ are general).

We blow-up the double locus $x_{0} = t = 0$. Let $E$ be the exceptional divisor, which is isomorphic to $P$. Let $C''$ denote the restriction to $E$ of the strict transform of $\cC$. On the affine patch $x_1y_0\neq0$, the equation of $C''$ is given by the smallest degree terms with respect to $x_{0}$ and $t$ in the equation for $\cC$. Note that $y_{1}$ is regarded as a constant during this computation. Thus, after homogenizing back, $C''$ is the $(2,2)$ curve on $E$ given by
\begin{equation*}
	y_{0}^{2}(x_{0} - t\lambda_{1})(x_{0} - t\lambda_{2})
	\lambda_{3}\lambda_{4} + 
	y_{1}^{2}(x_{0} - t\lambda_{5})(x_{0} -t\lambda_{6})
	\lambda_{7}\lambda_{8} = 0,
\end{equation*}
where $([x_0:t],[y_0:y_1])$ are the coordinates of $E$. The restriction of the strict transform of $\cL_{0}+\cL_{1}$ to $E$ also consists of the two lines $y_0=0$ and $y_1=0$. The resulting limit is described in Figure \ref{fig:typeageneral}. The central fiber is semi-log canonical, and one may check that 
\[
	K_P + D_P + \frac{1+\epsilon}{2}(L_0+L_1+C')
	\sim (-2,-2)+(1,0)+\frac{1+\epsilon}{2}(2,4) 
	=\epsilon(1,2),
\]
which is ample. By symmetry, this is enough to show that the limit is stable.

To conclude this example, observe that $\vec{\lambda}_{0}$ is discarded in the partial desingularization process. However, the study we carried out is preliminary to the next Example~\ref{typeasecondexample}.
\end{example}

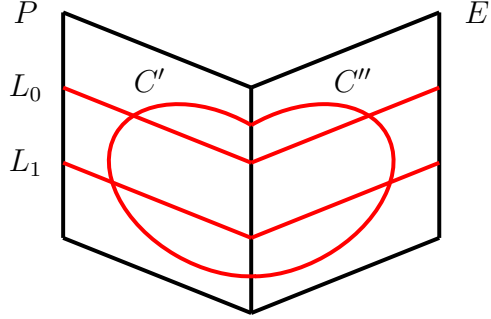
\begin{figure}[!ht]
\begin{tikzpicture}[scale=0.5]
	\draw[line width=1.5pt] (5,0) -- (5,6);
	\draw[line width=1.5pt] (0,2) -- (0,8);
	\draw[line width=1.5pt] (10,2) -- (10,8);

	\draw[line width=1.5pt] (5, 0) -- (10, 2);
	\draw[line width=1.5pt] (5,0) -- (0, 2);

	\draw[line width=1.5pt,red] (5, 2) -- (10, 4);
	\draw[line width=1.5pt,red] (5,2) -- (0, 4);
	
	\draw[line width=1.5pt,red] (5, 4) -- (10, 6);
	\draw[line width=1.5pt,red] (5,4) -- (0, 6);
	
	\draw[line width=1.5pt] (5, 6) -- (10, 8);
	\draw[line width=1.5pt] (5, 6) -- (0, 8);
	
	\draw[line width=1.5pt,rotate around={30:(5,5)},red] (5,3) [partial ellipse=-130:90:3cm and 2cm];
	\draw[line width=1.5pt,rotate around={-30:(5,5)},red] (5,3) [partial ellipse=90:315:3cm and 2cm];
	
	\node at (-1,8) {$P$};
	\node at (11,8) {$E$};
	\node at (2.3,6.2) {$C'$};
	\node at (7.7,6.2) {$C''$};
	\node at (-1,6) {$L_0$};
	\node at (-1,4) {$L_1$};
\end{tikzpicture}
\caption{Stable limit pair in Example~\ref{typeafirstexample}.}\label{fig:typeageneral}
\end{figure}

\begin{example}
\label{typeasecondexample}
Consider the one-parameter family of divisors $\cC$ on $P \times \Delta$ given by
\begin{equation*}
\begin{split}
	y_0^2(x_{0} - t\lambda_{1}x_1)(x_{0} - t\lambda_{2}x_1)&
	(tx_{0} - \lambda_{3}x_1)(tx_{0} - \lambda_{4}x_1) \\
	+ \;
	y_{1}^{2}&(x_{0} - t\lambda_{5}x_1)(x_{0} - t\lambda_{6}x_1)
	(tx_{0} - \lambda_{7}x_1)(tx_{0} -\lambda_{8}x_1) = 0,
\end{split}
\end{equation*}
whose special fiber $\cC_{0}$ is 
\[
	(\lambda_3\lambda_4y_{0}^{2}+\lambda_7\lambda_8y_{1}^{2})x_{0}^{2}x_{1}^{2} = 0.
\]
In this case, $\vec{\lambda}_{0}$ is a strictly semi-stable point with a closed orbit. Thus the normal direction $\vec{\lambda}_{t}$ corresponds to a stable point on the exceptional divisor of the partial desingularization. When $t = 0$, $\cB_{0}$ is the union of four distinct horizontal lines $y_{0} = 0$, $y_{1} = 0$, $y_{0} = \pm \sqrt{-\lambda_7\lambda_8/(\lambda_3\lambda_4)} y_{1}$, and two non-reduced vertical lines $x_{0} = 0$, $x_{1} = 0$ with multiplicity two. 

As we did in Example~\ref{typeafirstexample}, let $E_0$ (resp. $E_1$) be the exceptional divisor of the blow-up along $x_0=t=0$ (resp. $x_1=t=0$). Observe that these two exceptional divisors are isomorphic to $P$. Let $\widetilde{\cC}$ be the strict transform of $\cC$. On $E_0$, which has coordinates $([x_0:t],[y_0:y_1])$, the restriction of $\widetilde{\cC}$ has equation
\[
C_0 : y_0^2(x_0-t\lambda_1)(x_0-t\lambda_2)\lambda_3\lambda_4+y_1^2(x_0-t\lambda_5)(x_0-t\lambda_6)\lambda_7\lambda_8=0.
\]
Similarly, on $E_1$, which has coordinates $([t:x_1],[y_0:y_1])$, the restriction of $\widetilde{\cC}$ has equation
\[
C_1 : y_0^2(t-\lambda_3x_1)(t-\lambda_4x_1)+y_1^2(t-\lambda_7x_1)(t-\lambda_8x_1)=0.
\]
The limit pair is pictured in Figure \ref{fig:ssreplacementtypea}. On the central component, 
\[
	K_P+D_P+\frac{1+\epsilon}{2}B_P
	=(-2,-2)+(2,0)+\frac{1+\epsilon}{2}(0,4)=\epsilon(0,2),
\] 
which is not ample. Therefore, we need to contract $P$ in the central fiber horizontally, obtaining a surface with two components $E_0$ and $E_1$. The resulting stable pair is the same to that in Figure \ref{fig:typeageneral}. 
\end{example}

\begin{figure}[!ht]
\begin{tikzpicture}[scale=0.5]
	\draw[line width=1.5pt] (5,0) -- (5,6);
	\draw[line width=1.5pt] (10,0) -- (10,6);
	\draw[line width=1.5pt] (4.98,0) -- (10.02,0);
	\draw[line width=1.5pt] (5,6) -- (10,6);
	\draw[line width=1.5pt] (0,2) -- (0,8);
	\draw[line width=1.5pt] (15,2) -- (15,8);

	\draw[line width=1.5pt] (10, 0) -- (15, 2);
	\draw[line width=1.5pt] (5,0) -- (0, 2);

	\draw[line width=1.5pt,red] (10, 2) -- (15, 4);
	\draw[line width=1.5pt,red] (5,2) -- (0, 4);
	\draw[line width=1.5pt,red] (4.98,2) -- (10.02,2);
	\draw[line width=1.5pt,red] (4.98,4) -- (10.02,4);
	\draw[line width=1.5pt,red] (4.98,5) -- (10.02, 5);
	\draw[line width=1.5pt,red] (5,0.97) -- (10, 0.97);
	
	\draw[line width=1.5pt,red] (10, 4) -- (15, 6);
	\draw[line width=1.5pt,red] (5,4) -- (0, 6);
	
	\draw[line width=1.5pt] (10, 6) -- (15, 8);
	\draw[line width=1.5pt] (5, 6) -- (0, 8);
	
	\draw[line width=1.5pt,rotate around={30:(10,5)},red] (10,3) [partial ellipse=-133:90:3cm and 2cm];
	\draw[line width=1.5pt,rotate around={-30:(5,5)},red] (5,3) [partial ellipse=90:314:3cm and 2cm];
	
	\node at (-1,8) {$E_{0}$};
	\node at (16,8) {$E_{1}$};
	\node at (7.5, 7) {$P$};
	\node at (-1,6) {$L_0$};
	\node at (-1,4) {$L_1$};
\end{tikzpicture}
\caption{Semi-log canonical limit pair in Example~\ref{typeasecondexample}. The stable limit is obtained after contracting horizontally the middle component.}\label{fig:ssreplacementtypea}
\end{figure}
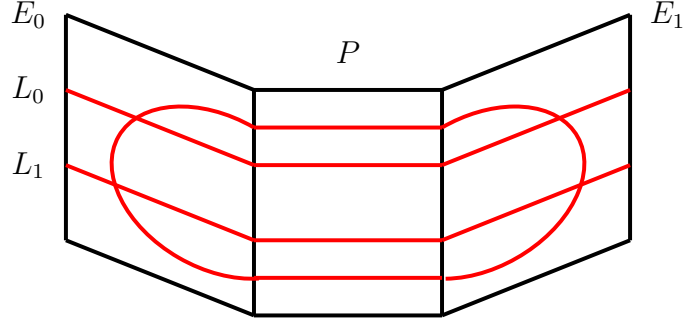

\begin{remark}\label{rmk:limitcovertypea}
Let us describe the double cover $\pi : X=X_{0}\cup X_{1}\rightarrow E_0\cup E_1$ of the degeneration in Example~\ref{typeasecondexample}. Consider the double cover $\pi_{0} : X_{0}\rightarrow E_{0}$, whose branch divisor is of class $(2,4)$ and has four isolated singularities of type $A_1$. We have that 
\[
	K_{X_{0}}=\pi_{0}^*(K_{E_{0}}+(1,2))=\pi_{0}^*(-1,0),
\] 
which implies that $-K_{X_{0}}$ is nef and $K_{X_{0}}^2=0$. From $\pi_{0}{}_*\cO_{X_{0}} \cong \cO_{X_{0}}\oplus \cO_{X_{0}}(-1, -2)$ we can argue that $\rH^{1}(\cO_{X_{0}}) = 0$. Moreover, by the projection formula, we have that
\[
\pi_{0}{}_*\cO_{X_0}(2K_{X_0})=\pi_{0}{}_*\pi_{0}^{*}\cO_{X_{0}}(-2, 0) \cong  \cO_{X_{0}}(-2, 0) \oplus \cO_{X_{0}}(-3, -2),
\]
so $\rH^0(\mathcal{O}_{X_{0}}(2K_{X_{0}}))=0$. Therefore, by Castelnuovo's rationality criterion, $X_{0}$ is rational. Note that $X_0$ has an elliptic fibration induced by the projection on $E_0$ given by $([x_0:t],[y_0:y_1])\mapsto[x_0:t]$. The same calculations work for $X_1$. Observe that the gluing locus $X_{0}\cap X_{1}$ is a genus one curve.
\end{remark}

In Example~\ref{typeasecondexample} we started with eight general points. Note that the same construction is valid for all the degenerate point configurations of interest, as we clarify in the following lemma. 

\begin{lemma}\label{lem:typealc}
Consider a semi-stable choice of $\lambda_1,\ldots,\lambda_8\in\mathbb{C}^8\subseteq(\mathbb{P}^1)^8$ such that $\lambda_{3}\lambda_{4}\lambda_{7}\lambda_{8}\neq0$, at least one of $\lambda_1,\lambda_2,\lambda_5,\lambda_6$ is different from the others, and at least one of $\lambda_3,\lambda_4,\lambda_7,\lambda_8$ is different from the others. Then the irreducible components $(E_i, \frac{1+\epsilon}{2}(C_i + L_{0}+L_{1})+D_{E_i})$ of the limit pair are stable, $i=0,1$.
\end{lemma}

\begin{remark}
\label{reasoningbehindrestrictionsonlambdas}
The reason why we exclude the cases $\lambda_{3}\lambda_{4}\lambda_{7}\lambda_{8}=0$ and $\lambda_1=\lambda_2=\lambda_5=\lambda_6$ is because, in the former case, $\vec{\lambda}_{0}$ is an unstable point configuration. In the latter case, $\vec{\lambda}_{t}$ is a curve along the strictly semi-stable locus in $(\PP^{1})^{8}$, and this case can be discarded in the partial desingularization process, as $\vec{\lambda}_{t}$ with $t \ne 0$ is unstable on the blow-up of $((\PP^{1})^{8})^{ss}$. For this same reason, we also exclude $\lambda_3=\lambda_4=\lambda_7=\lambda_8$.
\end{remark}

\begin{proof}[Proof of Lemma \ref{lem:typealc}]
Let us prove the claim for the pair corresponding to $i=0$, since the other one is analogous. We only have to prove that $(E_0, \frac{1+\epsilon}{2}(C_0 + L_{0}+L_{1})+D_{E_0})$ is log canonical. Up to symmetries, the significant cases are $\lambda_{1}, \lambda_{2}, \lambda_{5}, \lambda_{6}$ distinct (where $C_0$ is smooth), $\lambda_1=\lambda_2$, $\lambda_1=\lambda_5$, and $\lambda_1=\lambda_2=\lambda_5$. In each case we conclude that $C_0+L_{0}+L_{1}$ produces planar singularities that appear on Table \ref{tbl:degptconf} only, and on $D_{E_0}$ they do not have any double point (we omit these simple computations for brevity). So the pair $(E_0, \frac{1+\epsilon}{2}(C_0+L_{0}+L_{1})+D_{E})$ is log canonical.
\end{proof}


\subsection{Semi-stable points of \textsf{type b}}

\begin{example}
\label{typebfirstexample}
Consider the family of pairs $(P \times \Delta, \frac{1+\epsilon}{2}\cB) \to \Delta$ where $\cC$ on $P\times\Delta$ is given by
\begin{equation*}
\begin{split}
	y_0^2(x_{0} - t\lambda_{1}x_1)(x_{0} - t\lambda_{2}x_1)&
	(x_{0} - t\lambda_{3}x_1)(x_{0} - \lambda_{4}x_1)\\
	+ y_{1}^{2}&(x_{0} - t\lambda_{5}x_1)(x_{0} - \lambda_{6}x_1)
	(x_{0} - \lambda_{7}x_1)(x_{0} - \lambda_{8}x_1) = 0.
\end{split}
\end{equation*}
For $t \ne 0$ the pair is stable. For $t = 0$, the only non-log canonical singularity is at $x_{0} = y_{1} = 0$, and it is locally analytically four distinct concurrent lines with weight $\frac{1+\epsilon}{2}$. So we restrict to the affine patch $x_1y_0\neq 0$. Take the blow-up of $P \times \Delta$ at $t = x_{0} = y_{1} = 0$ and let $E\cong\PP^2$ be the exceptional divisor of the blow-up. Denote by $C_E$ the restriction to $E$ of the strict transform of $\cC$. On $E$, the divisor $C_E$ has equation
\begin{equation*}
	(x_{0} - t\lambda_{1})(x_{0} - t\lambda_{2})
	(x_{0} - t\lambda_{3})\lambda_{4} 
	+ y_{1}^{2}(x_{0} - t\lambda_{5})
	\lambda_{6}\lambda_{7}\lambda_{8} = 0, 
\end{equation*}
which a smooth cubic curve if $\vec{\lambda}$ is general, which we assumed. $C$ does not have an irreducible component $V(y_{1})$. If $h$ denotes the class of a line on $E$, then
\[
	K_E + D_{E} + \frac{1+\epsilon}{2}B_E
	= \left(-3 + 1 + \frac{1+\epsilon}{2}4\right)h = 2\epsilon h,
\]
which is ample. See the left hand side of Figure \ref{fig:stablereplacementtypeb} for the limit pair. 

Let $\pi : P'\rightarrow P$ be the blow-up of the central fiber, where the exceptional divisor is the conductor $D_{P'}$ and $B_{P'}$ is the strict transform of $\cB_0$. We have $K_{P'} = \pi^{*}(-2, -2) + D_{P'}$ and $B_{P'} = \pi^{*}(4, 4) - 4D_{P'}$. We can see that the divisor $K_{P'} + D_{P'} + \frac{1+\epsilon}{2}B_{P'}$ is not ample on $P'$. Consider the proper transform of the line $V(y_{1})$, whose class is $\pi^{*}(0, 1) - D_{P'}$. It is simple to check that the intersection number $(\pi^{*}(0, 1) - D_{P'})\cdot(K_{P'} + D_{P'} + \frac{1+\epsilon}{2}B_{P'})$ equals zero. The same computation yields the intersection with the proper transform of $V(x_{0})$ is also zero. After contracting these two lines, we obtain a stable limit with two irreducible components isomorphic to $\PP^{2}$, as in the right hand side of Figure \ref{fig:stablereplacementtypeb}. 
\end{example}

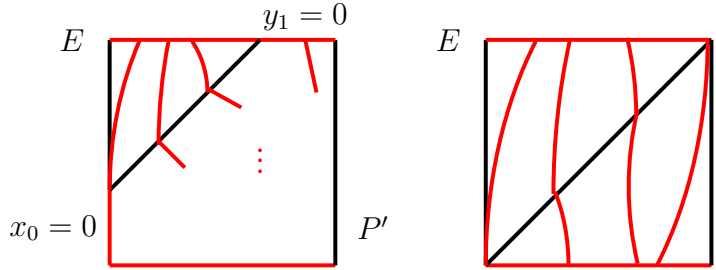
\begin{figure}[!ht]
\begin{tikzpicture}[scale=0.5]
	\draw[line width=1.5pt] (0,2) -- (0,6);
	\draw[line width=1.5pt,red] (0,0) -- (6,0);
	\draw[line width=1.5pt] (6,0) -- (6,6);
	\draw[line width=1.5pt,red] (0,6) -- (6,6);
	\draw[line width=1.5pt,red] (0,0) -- (0,2);
	\draw[line width=1.5pt] (0,2) -- (4,6);
	\draw[line width=1.5pt,red] (6,2) [partial ellipse=180:150:6 and 8];
	\draw[line width=1.5pt,red] (6.8,3.2) [partial ellipse=180:162:5.5 and 9];
	\draw[line width=1.5pt,red] (0.1, 4.6) [partial ellipse=0:35:2.5 and 2.5];
	\draw[line width=1.5pt,red] (1.3,3.3) -- (2, 2.6);
	\draw[line width=1.5pt,red] (2.6, 4.7) -- (3.5, 4.2);
	\draw[line width=1.5pt,red] (5.2, 6) -- (5.5, 4.6);
	\node at (-1,6) {$E$};
	\node at (7,1) {$P'$};
	\node at (5.2, 6.6) {$y_{1} = 0$};
	\node at (-1.5, 1) {$x_{0} = 0$};

	\draw[line width=1.5pt] (10,0) -- (10,6);
	\draw[line width=1.5pt,red] (10,0) -- (16,0);
	\draw[line width=1.5pt] (16,0) -- (16,6);
	\draw[line width=1.5pt,red] (10,6) -- (16,6);
	\draw[line width=1.5pt] (10,0) -- (16,6);
	\draw[line width=1.5pt,red] (16,0) [partial ellipse=180:141:6 and 9.5];
	\draw[line width=1.5pt,red] (16.89,1.9) [partial ellipse=180.5:156:5 and 10];
	\draw[line width=1.5pt,red] (10, 4) [partial ellipse=0:20:4 and 6];
	\draw[line width=1.5pt,red] (9.9,6) [partial ellipse=0:-39:6 and 9.5];
	\draw[line width=1.5pt,red] (17.79,2.1) [partial ellipse=161:200.1:4 and 6];
	\draw[line width=1.5pt,red] (10.1, 0) [partial ellipse=0:32:2.1 and 3.5];
	\node at (9,6) {$E$};
\end{tikzpicture}
\caption{Stable limit pair in Example~\ref{typebfirstexample}.}
\label{fig:stablereplacementtypeb}
\end{figure}

\begin{example}
\label{typebsecondexample}
Consider the family of pairs $(P \times \Delta, \frac{1+\epsilon}{2}\cB) \to \Delta$, where $\cC$ is given by 
\begin{equation*}
\begin{split}
	y_0^2(x_{0} - t\lambda_{1}x_1)(x_{0} - t\lambda_{2}x_1)&
	(x_{0} - t\lambda_{3}x_1)(tx_{0} - \lambda_{4}x_1)\\
	+ y_{1}^{2}&(x_{0} - t\lambda_{5}x_1)(tx_{0} - \lambda_{6}x_1)
	(tx_{0} - \lambda_{7}x_1)(tx_{0} - \lambda_{8}x_1) = 0.
\end{split}
\end{equation*}
Observe that the limit $8$-point configuration is a strictly semi-stable point of \textsf{type b}. On the central fiber, there are two non-log canonical singularities at $x_{0} = y_{1} = 0$ and at $x_{1} = y_{0} = 0$. By taking two blow-ups along those two points, which introduces the exceptional divisors $E_0$ and $E_1$ respectively, we have a configuration of three surfaces as in Figure \ref{fig:ssreplacementtypeb}. On $E_0$, the restriction of the strict transform of $\cC$ is given by
\[
C_0 : (x_0-t\lambda_1)(x_0-t\lambda_2)(x_0-t\lambda_3)\lambda_4+y_1^2(x_0-t\lambda_5)\lambda_6\lambda_7\lambda_8=0.
\]
On $E_1$, the restriction of the strict transform of $\cC$ is given by
\[
C_1 : y_0^2(t-\lambda_4x_1)+(t-\lambda_6x_1)(t-\lambda_7x_1)(t-\lambda_8x_1)=0.
\]
The central component $P''$ is a ruled surface and 
\[
\begin{split}
	K_{P''} + D_{P''} + \frac{1+\epsilon}{2}B_{P''}
	= &\;(\pi^{*}(-2, -2) + D_{E_0} + D_{E_1})\\
	& + D_{E_{0}} + D_{E_{1}} + \frac{1+\epsilon}{2}
	(\pi^{*}(4, 4) - 4D_{E_0} - 4D_{E_1})\\
	= &\;\epsilon(\pi^{*}(2, 2) - 2D_{E_0} - 2D_{E_1})
\end{split}
\]
is not ample. $P''$ is contracted diagonally, thus the resulting stable limit is the same as the picture on the right hand side in Figure \ref{fig:stablereplacementtypeb}.
\end{example}

\begin{figure}[!ht]
\begin{tikzpicture}[scale=0.5]
	\draw[line width=1.5pt] (0,2) -- (0,6);
	\draw[line width=1.5pt,red] (0,0) -- (6,0);
	\draw[line width=1.5pt] (6,0) -- (6,4);
	\draw[line width=1.5pt,red] (0,6) -- (6,6);
	\draw[line width=1.5pt,red] (0,0) -- (0,2);
	\draw[line width=1.5pt] (0,2) -- (4,6);
	\draw[line width=1.5pt] (2, 0) -- (6, 4);
	\draw[line width=1.5pt,red] (6, 4) -- (6,6);
	\draw[line width=1.5pt,red] (6,2) [partial ellipse=180:150:6 and 8];
	\draw[line width=1.5pt,red] (6.8,3.2) [partial ellipse=180:162:5.5 and 9];
	\draw[line width=1.5pt,red] (0.1, 4.6) [partial ellipse=0:35:2.5 and 2.5];
	\draw[line width=1.5pt,red] (0,4) [partial ellipse=0:-30:6 and 8];
	\draw[line width=1.5pt,red] (-0.8,2.8) [partial ellipse=0:-18:5.5 and 9];
	\draw[line width=1.5pt,red] (5.9, 1.4) [partial ellipse=180:215:2.5 and 2.5];
	\draw[line width=1.5pt,red] (1.3,3.3) -- (3.4, 1.3);
	\draw[line width=1.5pt,red] (2.6, 4.7) -- (4.7, 2.7);
	\node at (-1,6) {$E_{0}$};
	\node at (7,1) {$E_{1}$};
	\node at (3, 3) {$P''$};
\end{tikzpicture}
\caption{Semi-log canonical limit pair in Example~\ref{typebsecondexample}. The stable limit is obtained after contracting diagonally the middle component.}
\label{fig:ssreplacementtypeb}
\end{figure}
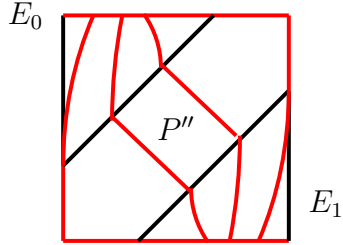

\begin{remark}
For $i=0,1$, the double cover $\pi_{i} : X_{i} \to E_i$ of an irreducible component $E_i \cong \PP^{2}$ is branched along a reducible quartic curve $B_{i}$ with three nodal singularities. Since $K_{X_{i}} = \pi_{i}^{*}K_{E_i}+ \frac{1}{2}\pi_{i}^{*}B_{i}  = -\pi_{i}^{*}h$ where $h$ is the class of a line, we have that $K_{X_i}$ is anti-ample. Since $X_{i}$ has only three $A_1$ singularities, the minimal resolution of $X_{i}$ is a weak del Pezzo surface of degree $2$. The gluing locus $X_0\cap X_1$ is a genus one curve.
\end{remark}

The same construction is valid for all the degenerate point configurations of interest in view of the partial desingularization. The reasoning behind the restrictions on the $\lambda_i$ in the next lemma is analogous to what we already explained in Remark~\ref{reasoningbehindrestrictionsonlambdas}.

\begin{lemma}\label{lem:typeblc}
Consider a semi-stable choice of $\lambda_1,\ldots,\lambda_8\in\mathbb{C}^8\subseteq(\mathbb{P}^1)^8$ such that $\lambda_{4}\lambda_{6}\lambda_{7}\lambda_{8}\neq0$, at least one of $\lambda_1,\lambda_2,\lambda_3,\lambda_5$ is different from the others, and at least one of $\lambda_4,\lambda_6,\lambda_7,\lambda_8$ is different from the others. Then the irreducible components $(E_0, \frac{1+\epsilon}{2}(C_0 + L_{1})+D_{E_0}),(E_1, \frac{1+\epsilon}{2}(C_1 + L_{0})+D_{E_1})$ of the limit pair are stable.
\end{lemma}

\begin{proof}
By symmetry, it is enough to check the log canonicity of $(E_0, \frac{1+\epsilon}{2}(C_0 + L_{1})+D_{E_0})$. By checking degenerate point configurations, under the assumption that one of $\lambda_{1}, \lambda_{2}, \lambda_{3}, \lambda_{5}$ is distinct from the others, one may conclude that $(C_0+L_{1})$ has only singularities that appear on Table \ref{tbl:degptconf}, and along $D_{E} = V(t)$ it has no singularities at all. So the pair $(E_0, \frac{1+\epsilon}{2}(C_0 + L_{1})+D_{E_0})$ is log canonical. 
\end{proof}


\subsection{Semi-stable points of \textsf{type c}}
\label{stablereplacementtypec}

\begin{example}
\label{typecfirstexample}
Consider the one-parameter family of divisors $\cC$ on $\cP := P\times\Delta$ given by
\begin{equation*}
	y_0^2\prod_{i=1}^4(x_0-t\lambda_ix_1)+
	y_1^2\prod_{i=5}^8(x_0-\lambda_ix_1)=0.
\end{equation*}
The divisor $\cB_0$ has five multiple points, but only $([0:1],[1:0])$ is not a simple double point. In the affine patch $x_1y_0\neq0$, locally at $(0,0)$, the singularity of $\cB_0$ is isomorphic to $y_1(x_0^4+y_1^2)=0$, which is not log canonical for the weight $(1+\epsilon)/2$. To find the stable replacement, let us restrict to the affine patch $x_1y_0\neq0$, so that the equation for $\cB$ becomes
\[
	y_{1}\left(\prod_{i=1}^4(x_0-t\lambda_i)+
	y_1^2\prod_{i=5}^8(x_0-\lambda_i)\right)=0.
\]
We perform the following birational modifications (see Figure~\ref{case5}).
\begin{enumerate}
\item Let $\cP'\rightarrow\cP$ be the weighted blow-up at $(t,x_0,y_1) = (0,0,0)$ with weights $(1,1,2)$, which is the blow-up of the ideal $(t^2,tx_0,x_0^2,y_1)$. The exceptional divisor $E_1$ is isomorphic to the weighted projective plane $\PP(1,1,2)$, which is isomorphic to $\mathbb{F}_2^0$, a cone over a smooth conic. This weighted blow-up introduces an $A_1$ singularity on each irreducible component of $\cP'_{0}$. Let $\cB'$ be the strict transform of the divisor $\cB$. Observe that the equation for $\cB'_{E_{1}}$, which is the smallest degree part with respect to $x_0, y_1, t$ of the equation for $\cB$, is given by
\begin{equation*}
	y_{1}\left(\prod_{i=1}^4(x_0-t\lambda_i)+
	y_1^2\prod_{i=5}^8\lambda_i\right)=0.
\end{equation*}
It is the union of a $2$-section and a section of $\mathbb{F}_2^0$. 

\item The strict transform of the line $L_1=V(y_1)$ on $P$ intersects $K_{\cP_0'}+\frac{1+\epsilon}{2}\cB_0'$ negatively. We flip this curve by first blowing it up, introducing an exceptional divisor $E_2$ isomorphic to $\mathbb{F}_1$.

\item Blow-up the strict transform of $L_1$ on $E_2$ introducing the exceptional divisor $E_3$ isomorphic to $\mathbb{F}_0$.

\item Contract $E_3$ along the ruling intersecting $E_2$ in the exceptional divisor. The strict transform of $E_2$ is now isomorphic to $\PP^2$.

\item Contract the strict transform of $E_2$ to a point. This introduces an $A_1$ singularity on the strict transform of $P$ and on the strict transform of $E_1$.
\end{enumerate}

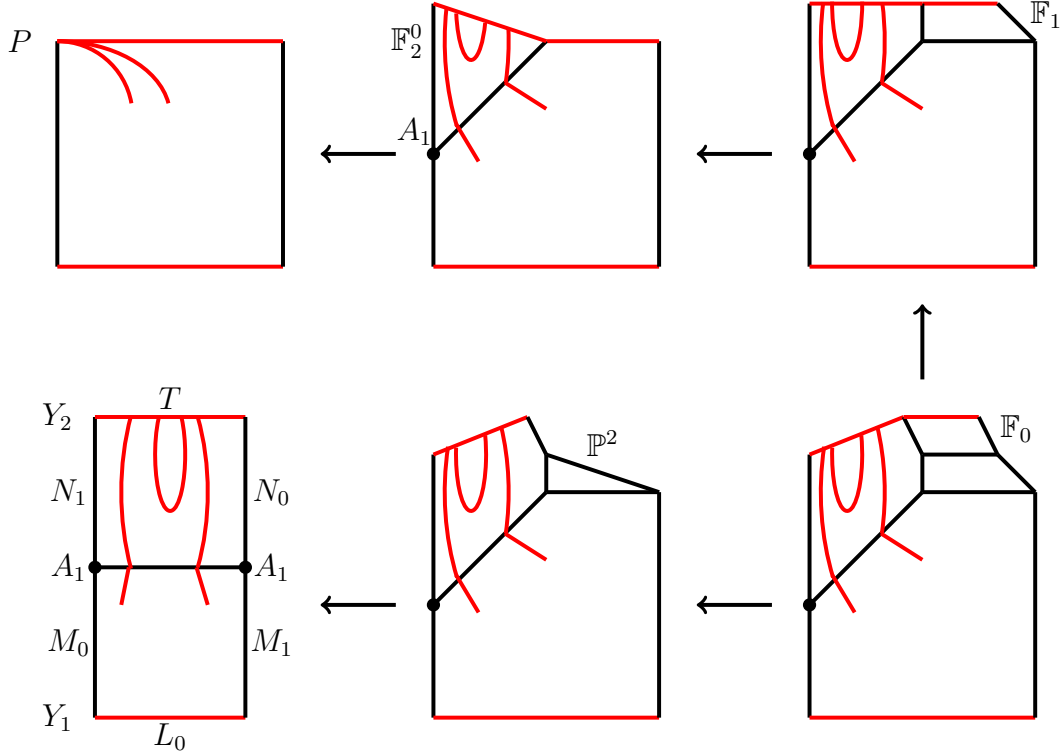
\begin{figure}[!hbt]
\begin{tikzpicture}[scale=0.5]
	\node at (-1,6) {$P$};
	\draw[line width=1.5pt] (0,0) -- (0,6);
	\draw[line width=1.5pt,red] (0,0) -- (6,0);
	\draw[line width=1.5pt] (6,0) -- (6,6);
	\draw[line width=1.5pt,red] (0,6) -- (6,6);
	\draw[line width=1.5pt,red] (0,4) [partial ellipse=10:90:3 and 2];
	\draw[line width=1.5pt,red] (0,4) [partial ellipse=10:90:2 and 2];

	\draw[->][line width=1.5pt] (9, 3) -- (7, 3);

	\node at (9.3,6) {$\mathbb{F}_2^0$};
	\draw[line width=1.5pt] (10,0) -- (10,3);
	\draw[line width=1.5pt,red] (10,0) -- (16,0);
	\draw[line width=1.5pt] (16,0) -- (16,6);
	\draw[line width=1.5pt,red] (13,6) -- (16,6);
	\draw[line width=1.5pt] (10,3) -- (13,6);
	\fill (10,3) circle (5pt);
	\node at (9.5,3.6) {$A_1$};
	\draw[line width=1.5pt] (10,3) -- (10,7);
	\draw[line width=1.5pt,red] (10,7) -- (13,6);
	\draw[line width=1.5pt,red] (13.3, 6) [partial ellipse=207:170:3 and 5];
	\draw[line width=1.5pt,red] (10, 6) [partial ellipse=-16:5:2 and 4];
	\draw[line width=1.5pt,red] (11, 7) [partial ellipse=190:340:0.4 and 1.5];
	\draw[line width=1.5pt,red] (11.9,4.9) -- (13,4.2);
	\draw[line width=1.5pt,red] (10.6,3.8) -- (11.2,2.8);

	\draw[->][line width=1.5pt] (19, 3) -- (17, 3);

	\node at (26.3,6.7) {$\mathbb{F}_1$};
	\draw[line width=1.5pt] (20,0) -- (20,3);
	\draw[line width=1.5pt,red] (20,0) -- (26,0);
	\draw[line width=1.5pt] (26,0) -- (26,6);
	\draw[line width=1.5pt] (23,6) -- (26,6);
	\draw[line width=1.5pt,red] (23,7) -- (25,7);
	\draw[line width=1.5pt] (23,7) -- (23,6);
	\draw[line width=1.5pt] (25,7) -- (26,6);
	\draw[line width=1.5pt] (20,3) -- (23,6);
	\fill (20,3) circle (5pt);
	\draw[line width=1.5pt] (20,3) -- (20,7);
	\draw[line width=1.5pt,red] (20,7) -- (23,7);
	\draw[line width=1.5pt,red] (23.3, 6) [partial ellipse=207:168:3 and 5];
	\draw[line width=1.5pt,red] (20, 6) [partial ellipse=-16:15:2 and 4];
	\draw[line width=1.5pt,red] (21, 7) [partial ellipse=178:362:0.4 and 1.5];
	\draw[line width=1.5pt,red] (21.9,4.9) -- (23,4.2);
	\draw[line width=1.5pt,red] (20.6,3.8) -- (21.2,2.8);
	
	\draw[->][line width=1.5pt] (23, -3) -- (23, -1);

	\node at (25.5,7-12.3+1) {$\mathbb{F}_0$};
	\draw[line width=1.5pt] (20,-12) -- (20,3-12);
	\draw[line width=1.5pt,red] (20,-12) -- (26,-12);
	\draw[line width=1.5pt] (26,-12) -- (26,6-12);
	\draw[line width=1.5pt] (23,6-12) -- (26,6-12);
	\draw[line width=1.5pt] (23,7-12) -- (25,7-12);
	\draw[line width=1.5pt] (23,7-12) -- (23,6-12);
	\draw[line width=1.5pt] (25,7-12) -- (26,6-12);
	\draw[line width=1.5pt] (20,3-12) -- (23,6-12);
	\fill (20,3-12) circle (5pt);
	\draw[line width=1.5pt] (20,3-12) -- (20,7-12);
	\draw[line width=1.5pt,red] (20,7-12) -- (23-0.5,7-12+1);
	\draw[line width=1.5pt] (23-0.5,7-12+1) -- (23,7-12);
	\draw[line width=1.5pt,red] (23-0.5,7-12+1) -- (25-0.5,7-12+1);
	\draw[line width=1.5pt] (25-0.5,7-12+1) -- (25,7-12);
	\draw[line width=1.5pt,red] (23.3, -6) [partial ellipse=207:166:3 and 5];
	\draw[line width=1.5pt,red] (20, -6) [partial ellipse=-16:26:2 and 4];
	\draw[line width=1.5pt,red] (21, -5) [partial ellipse=173:380:0.4 and 1.5];
	\draw[line width=1.5pt,red] (21.9,4.9-12) -- (23,4.2-12);
	\draw[line width=1.5pt,red] (20.6,3.8-12) -- (21.2,2.8-12);

	\draw[->][line width=1.5pt] (19, 3-12) -- (17, 3-12);

	\draw[line width=1.5pt] (10,-12) -- (10,3-12);
	\draw[line width=1.5pt,red] (10,-12) -- (16,-12);
	\draw[line width=1.5pt] (16,-12) -- (16,6-12);
	\draw[line width=1.5pt] (13,6-12) -- (16,6-12);
	\draw[line width=1.5pt] (13,7-12) -- (13,6-12);
	\draw[line width=1.5pt] (13,7-12) -- (16,6-12);
	\node at (14.5,7-12+0.3) {$\PP^2$};
	\draw[line width=1.5pt] (10,3-12) -- (13,6-12);
	\fill (10,3-12) circle (5pt);
	\draw[line width=1.5pt] (10,3-12) -- (10,7-12);
	\draw[line width=1.5pt,red] (10,7-12) -- (13-0.5,7-12+1);
	\draw[line width=1.5pt] (13-0.5,7-12+1) -- (13,7-12);
	\draw[line width=1.5pt,red] (13.3, -6) [partial ellipse=207:166:3 and 5];
	\draw[line width=1.5pt,red] (10, -6) [partial ellipse=-16:26:2 and 4];
	\draw[line width=1.5pt,red] (11, -5) [partial ellipse=173:380:0.4 and 1.5];
	\draw[line width=1.5pt,red] (11.9,4.9-12) -- (13,4.2-12);
	\draw[line width=1.5pt,red] (10.6,3.8-12) -- (11.2,2.8-12);

	\draw[->][line width=1.5pt] (9, 3-12) -- (7, 3-12);

	\draw[line width=1.5pt] (0+1,-12) -- (0+1,-12+4);
	\draw[line width=1.5pt] (0+1,4-12) -- (4+1,4-12);
	\draw[line width=1.5pt,red] (0+1,-12) -- (4+1,-12);
	\draw[line width=1.5pt] (4+1,-12) -- (4+1,4-12);
	\draw[line width=1.5pt] (0+1,-12+4) -- (0+1,-12+4+4);
	\draw[line width=1.5pt,red] (0+1,4-12+4) -- (4+1,4-12+4);
	\draw[line width=1.5pt] (4+1,-12+4) -- (4+1,4-12+4);
	\draw[line width=1.5pt,red] (4.65, -6) [partial ellipse=204:156:3 and 5];
	\draw[line width=1.5pt,red] (2, -6) [partial ellipse=-31:30:2 and 4];
	\draw[line width=1.5pt,red] (3, -5) [partial ellipse=140:400:0.4 and 1.5];
	\draw[line width=1.5pt,red] (1.91,-8) -- (1.71,-9);
	\draw[line width=1.5pt,red] (3.71,-8) -- (4.01,-9);
	\fill (0+1,4-12) circle (5pt);
	\fill (4+1,4-12) circle (5pt);
	\node at (4.7+1,4-12) {$A_1$};
	\node at (0.3,4-12) {$A_1$};
	\node at (0, -4) {$Y_{2}$};
	\node at (0, -12) {$Y_{1}$};
	\node at (0.3, -6) {$N_1$};
	\node at (5.7, -6) {$N_0$};
	\node at (3, -3.5) {$T$};
	\node at (3, -12.5) {$L_{0}$};
	\node at (0.3, -10) {$M_{0}$};
	\node at (5.7, -10) {$M_{1}$};
\end{tikzpicture}
\caption{Stable replacement in Example~\ref{typecfirstexample}.}
\label{case5}
\end{figure}

The central fiber is log canonical (it was already log canonical at step (1) -- the modification that followed did not make the singularities of the pair worse), so we only need to check that the numerical condition is satisfied. Denote by $Y_1,Y_2$ the two irreducible components of the central fiber, where $Y_1$ is the strict transform of $P$. Let $B_{Y_{i}}$ be the restriction of the strict transform of $\cB$ to $Y_{i}$ . We want to show that $K_{Y_i}+D_{Y_i}+\frac{1+\epsilon}{2}B_{Y_i}$ is ample for $i=1,2$.

\begin{enumerate}
\item On $Y_1$, let $M_0, M_1$ be the two vertical boundary lines, where $M_0^2=-\frac{1}{2}$ and $M_1^2=\frac{1}{2}$. Note that $L_0$ is the horizontal boundary line disjoint from $Y_{2}$. It is a straightforward exercise to compute the following intersection numbers (observe that $L_0$ is contained in the support of $B_{Y_1}$):
\begin{displaymath}
\begin{array}{|c|c|c|c|c|c|}
\hline
\cdot & M_0 & M_1 & L_0 & D_{Y_1} & B_{Y_1}\\
\hline
M_0&-1/2&0&1&1/2&1\\
\hline
M_1&&1/2&1&1/2&3\\
\hline
L_0&&&0&0&4\\
\hline
D_{Y_1}&&&&0&2\\
\hline
\end{array}
\end{displaymath}
Moreover, $K_{Y_1}=-M_0-M_1-L_0-D_{Y_1}$. One may check that $K_{Y_1}+D_{Y_1}+\frac{1+\epsilon}{2}B_{Y_1}$ intersects positively with $M_0$, $M_1$, $L_0$, and $D_{Y_1}$. Since $Y_1$ is a toric variety, this implies $K_{Y_1}+D_{Y_1}+\frac{1+\epsilon}{2}B_{Y_1}$ is ample.

\item On $Y_2$, let $N_0,N_1$ be the two vertical boundary lines, where $N_0^2=-\frac{1}{2}$ and $N_1^2=\frac{1}{2}$. Denote by $T$ the top boundary line. We have the following intersection numbers:
\begin{displaymath}
\begin{array}{|c|c|c|c|c|c|}
\hline
\cdot&N_0&N_1&T&D_{Y_2}&B_{Y_2}\\
\hline
N_0&-1/2&0&1&1/2&1\\
\hline
N_1&&1/2&1&1/2&3\\
\hline
T&&&0&0&4\\
\hline
D_{Y_2}&&&&0&2\\
\hline
\end{array}
\end{displaymath}
Then $K_{Y_2}=-N_0-N_1-T-D_{Y_2}$ and $K_{Y_2}+D_{Y_2}+\frac{1+\epsilon}{2}B_{Y_2}$ is ample because it intersects positively with $N_0,N_1,T,D_{Y_2}$.
\end{enumerate}

Note that two intersection matrices for $Y_{1}$ and $Y_{2}$ are same. Indeed, they are isomorphic toric surfaces. But the non-toric divisors $B_{Y_{1}}$ and $B_{Y_{2}}$ are in general different.
\end{example}

\begin{example}
\label{typecsecondexample}
Finally, consider the one-parameter family of divisors $\cC$ on $\cP := P\times\Delta$ given by
\begin{equation*}
	y_0^2\prod_{i=1}^4(x_0-t\lambda_ix_1)+
	y_1^2\prod_{i=5}^8(tx_0-\lambda_ix_1)=0.
\end{equation*}
We have that $\left(\cP_0,\frac{1+\epsilon}{2}\cB_0\right)$ is not log canonical at $([0:1],[1:0]),([1:0],[0:1])$ where the divisor is locally isomorphic to $y(x^4+y^2)$. To obtain the stable replacement we may repeat the procedure in Example~\ref{typecfirstexample} for both singularities. Let $F_0$ be the exceptional divisor of the weighted blow-up at $([0:1],[1:0])$. Then the restriction to $F_0$ of the strict transform of $\cC$ has equation
\[
C_0 : \prod_{i=1}^4(x_0-t\lambda_i)+y_1^2\prod_{i=5}^8\lambda_i=0.
\]
Similarly, if $F_1$ be the exceptional divisor of the weighted blow-up at $([1:0],[0:1])$, then the restriction to $F_1$ of the strict transform of $\cC$ has equation
\[
C_1 : y_0^2+\prod_{i=5}^8(t-\lambda_ix_1)=0.
\]

Continuing with the flip described in Example~\ref{typecfirstexample} (which we perform once for each of the two singularities), the semi-stable replacement has three irreducible components: let $Y_{1}$ be the strict transform of $P$ and let $Y_2,Y_3$ be the strict transforms of the exceptional divisors of the two exceptional divisors $F_0,F_1$ respectively (see left picture in Figure \ref{fig:sslimit6}). Observe that $Y_1$ has four $A_1$ singularities and $Y_2$ and $Y_3$ have two $A_1$ singularities each. Let $B_{Y_{i}}$ be the restriction to $Y_{i}$ of the strict transform of $\cB$.

\begin{figure}[!ht]
\begin{tikzpicture}[scale=0.5]
	\draw[line width=1.5pt] (1,-4) -- (1,8);
	\draw[line width=1.5pt] (1,4) -- (5,4);
	\draw[line width=1.5pt,red] (1,-4) -- (5,-4);
	\draw[line width=1.5pt] (5,-4) -- (5,8);
	\draw[line width=1.5pt,red] (1,8) -- (5,8);
	\draw[line width=1.5pt] (1,0) -- (5,0);
	\draw[line width=1.5pt,red] (4.6, 6) [partial ellipse=204:156:3 and 5];
	\draw[line width=1.5pt,red] (2, 6) [partial ellipse=-31:30:2 and 4];
	\draw[line width=1.5pt,red] (3, 7) [partial ellipse=140:400:0.4 and 1.5];
	\draw[line width=1.5pt,red] (1.85,4) -- (1.85,0);
	\draw[line width=1.5pt,red] (3.73,4) -- (3.73,0);
	\draw[line width=1.5pt,red] (4.6, -2) [partial ellipse=204:156:3 and 5];
	\draw[line width=1.5pt,red] (2, -2) [partial ellipse=-30:30.5:2 and 4];
	\draw[line width=1.5pt,red] (2.6, -3) [partial ellipse=-40:220:0.4 and 1.5];
	\fill (1,4) circle (5pt);
	\fill (5,4) circle (5pt);
	\fill (5,0) circle (5pt);	
	\fill (1,0) circle (5pt);
	\node at (5.7,4) {$A_1$};
	\node at (0.3,4) {$A_1$};
	\node at (0.3,0) {$A_1$};
	\node at (5.7,0) {$A_1$};
	\node at (0, 6) {$Y_{2}$};
	\node at (0, -2) {$Y_{3}$};
	\node at (3, 2) {$Y_{1}$};
	\node at (0.3, 2) {$M_{0}$};
	\node at (5.7, 2) {$M_{1}$};

	\draw[line width=1.5pt] (9,-2) -- (9,6);
	\draw[line width=1.5pt] (9,6) -- (13,6);
	\draw[line width=1.5pt,red] (9,-2) -- (13,-2);
	\draw[line width=1.5pt] (13,-2) -- (13,6);
	\draw[line width=1.5pt,red] (9,6) -- (13,6);
	\draw[line width=1.5pt] (9,2) -- (13,2);
	\draw[line width=1.5pt,red] (12.6, 4) [partial ellipse=204:156:3 and 5];
	\draw[line width=1.5pt,red] (10, 4) [partial ellipse=-30:30:2 and 4];
	\draw[line width=1.5pt,red] (11, 5) [partial ellipse=140:400:0.4 and 1.5];
	\draw[line width=1.5pt,red] (12.6, 0) [partial ellipse=204:156:3 and 5];
	\draw[line width=1.5pt,red] (10, 0) [partial ellipse=-30:30.5:2 and 4];
	\draw[line width=1.5pt,red] (10.6, -1) [partial ellipse=-40:220:0.4 and 1.5];
	\fill (13,2) circle (5pt);	
	\fill (9,2) circle (5pt);
	\node at (8.3,2) {$A_1$};
	\node at (13.7,2) {$A_1$};
	\node at (8, 4) {$Y_{2}$};
	\node at (8, 0) {$Y_{3}$};
\end{tikzpicture}
\caption{On the left, the semi-log canonical limit pair in Example~\ref{typecsecondexample}, and on the right, the corresponding stable limit obtained after contracting the middle component.}
\label{fig:sslimit6}
\end{figure}
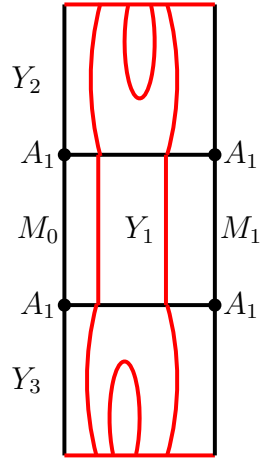

On $Y_1$, let $M_i$ be the strict transform of the line $x_i=0$, $i=0,1$. The conductor divisor $D_{Y_1}$ consists of the two irreducible components $D_{12}:=Y_1\cap Y_2$ and $D_{13}:=Y_1\cap Y_3$. We have the following intersection numbers:
\begin{displaymath}
\begin{array}{|c|c|c|c|c|c|}
\hline
\cdot&M_0&M_1&D_{12}&D_{13}&B_{Y_1}\\
\hline
M_0&0&0&1/2&1/2&0\\
\hline
M_1&&0&1/2&1/2&0\\
\hline
D_{12}&&&0&0&2\\
\hline
D_{13}&&&&0&2\\
\hline
\end{array}
\end{displaymath}
One can compute that $K_{Y_1}=-2M_0-2D_{13}$. Therefore, $K_{Y_1} + D_{Y_1} + \frac{1+\epsilon}{2}B_{Y_1}$ intersects $M_0$ and $M_1$ trivially and $Y_{1}$ can be contracted vertically to obtain a stable pair. The resulting pair is the same as the right picture in Figure~\ref{fig:sslimit6}.
\end{example}

\begin{remark}
We compute the double cover of the degeneration in Figure \ref{fig:sslimit6} after contracting the middle component $Y_1$. Let $\pi_{2} : X_{2} \to Y_{2}$ be the double cover (analogous considerations hold for $Y_3$). Using that $B_{Y_2} = 4N_0 + 3T$ and $K_{Y_{2}} = -(N_{0}+N_{1}+T+D_{Y_2}) = -2N_0-2T$ (notice that $Y_2$ has Picard number $2$), and following the same strategy as in Remark \ref{rmk:limitcovertypea}, by Castelnuovo's rationality criterion one can conclude that $X_{2}$ is a rational surface with $K_{X_2}^2=0$, which also comes with an elliptic fibration. The limit stable surface is the union of two rational surfaces glued along a genus one curve. 
\end{remark}

As in \textsf{type a} and \textsf{type b}, we may generalize this computation to all the degenerate point configurations of interest. 

\begin{lemma}\label{lem:typeclc}
Consider a semi-stable choice of $\lambda_1,\ldots,\lambda_8\in\mathbb{C}^8\subseteq(\mathbb{P}^1)^8$ such that $\lambda_{5}\lambda_{6}\lambda_{7}\lambda_{8}\neq0$, at least one of $\lambda_1,\lambda_2,\lambda_3,\lambda_4$ is different from the others, and at least one of $\lambda_5,\lambda_6,\lambda_7,\lambda_8$ is different from the others. Then the irreducible components $(Y_2, \frac{1+\epsilon}{2}B_{Y_2}+D_{Y_2}),(Y_3, \frac{1+\epsilon}{2}B_{Y_3}+D_{Y_3})$ of the limit pair are stable.
\end{lemma}

\begin{proof}
To check that the pairs are stable, it is sufficient to check the singularities of the pairs corresponding to the exceptional divisors of the first two weighted blow-ups, because the flip, which was described in Example~\ref{typecfirstexample}, does not produce any new singularities on the curve. With reference to the notation introduced in Example~\ref{typecsecondexample}, we have to show that $(F_{0}, \frac{1+\epsilon}{2}(C_0+L_{1})+D_{F_0})$ and $(F_{1}, \frac{1+\epsilon}{2}(C_1+L_{0})+D_{F_1})$ are log canonical. We only show this for the former pair since the proof for the latter is analogous.

Note that from the equation of $C_0$ in Example~\ref{typecsecondexample}, it is clear that $C_0$ does not intersect two coordinate points $V(t, x_{0})$ and $V(t, y_{1})$. In particular, $C_{0}+L_{1}$ is supported  on the smooth locus of $F_0$. If all $\lambda_{i}$'s for $i \le 4$ are distinct, then $C_{0}$ is nonsingular and intersects with $L_{1}$ and $D_{F_0}$ transversally. By checking the allowed collisions of the $\lambda_{i}$ with $i \le 4$, one can conclude that $(C_0+L_{1})$ has only singularities on Table \ref{tbl:degptconf}, and it intersects with $D_{F_0}$ at three points transversally. Thus the pair is log canonical. 
\end{proof}

\begin{remark}
Let $\cX^0\rightarrow\Delta\setminus\{0\}$ be a proper flat family of smooth K3 surfaces with a purely non-symplectic automorphism of order four and a $U(2)\oplus D_4^{\oplus2}$ lattice polarization as in \S\,\ref{kondo'sconstruction}. Let $\cX$ be a semistable completion of $\cX^0$ over $\Delta$ with $\cX$ smooth and $K_{\cX}\sim0$ (see \cite{Kul77,PP81}). If the central fiber $\cX_0$ is not smooth, then $\cX_0$ has Kulikov type II according to \cite[Theorem II]{Kul77}. This follows from the study of degenerations we carried out in \S\,\ref{explicitcalculationofstablereplacement}, together with \cite[Theorem 7.4]{Sch16}.
\end{remark}

\begin{remark}
Motivated by the study of the KSBA compactification of the moduli space of K3 surfaces of degree two, in \cite{AT21}, Alexeev and Thompson introduced classes of combinatorially defined surfaces called $ADE$ surfaces and $\widetilde{A}\widetilde{D}\widetilde{E}$ surfaces \cite[\S\,3]{AT21}. These are log canonical non-klt del Pezzo surfaces with reduced boundary. An $ADE$ (resp. $\widetilde{A}\widetilde{D}\widetilde{E}$) cover is a double cover of an $ADE$ (resp. $\widetilde{A}\widetilde{D}\widetilde{E}$) surface.

The three K3 degenerations of \textsf{type a}, \textsf{b}, and \textsf{c} in this section are reducible surfaces whose irreducible components are of type II according to \cite[Lemma~2.5]{AT21}. With reference to the notation introduced in \cite[\S\,3A]{AT21}, before taking the $\ZZ_{2}$-cover, we have that for a degeneration of \textsf{type a} each irreducible component is a $\widetilde{D}_{8}$ surface. For a \textsf{type b} degeneration each irreducible component is an $\widetilde{E}_{7}$ surface. Finally, for a \textsf{type c} degeneration, the first weighted blow-up in Figure~\ref{case5} produces a component which is an $\widetilde{E}_{8}^{-}$ surface. The next four blow-ups/downs correspond to the `priming' operation in \cite[Definition~3.13]{AT21}. The resulting degeneration has two irreducible components which are $\widetilde{E}_{8}^{+}$ surfaces. 

A stable $ADE$ (resp. $\widetilde{A}\widetilde{D}\widetilde{E}$) cover is obtained by gluing $ADE$ (resp. $\widetilde{A}\widetilde{D}\widetilde{E}$) covers along boundaries. Alexeev and Thompson constructed moduli spaces of $ADE$ covers \cite[Theorem C]{AT21}. Moduli spaces of $\widetilde{A}\widetilde{D}\widetilde{E}$ covers, which are related to our situation, are not completely studied yet. 
\end{remark}


\section{Proof of the main theorem}
\label{proofofmainthm}

We are ready to prove Theorem \ref{thm:mainthm}. Let $(\cY, \frac{1+\epsilon}{2}\cB) \to \bU$ as in the beginning of \S\,\ref{explicitcalculationofstablereplacement}. The first step is the following statement:

\begin{proposition}
\label{extendedfamilyonkirwan'sblowup}
There is a flat family of stable pairs $(\widetilde{\cY}, \frac{1+\epsilon}{2}\widetilde{\cB}) \to \bX_{1}^{s}$ which is an extension of $(\cY, \frac{1+\epsilon}{2}\cB) \to \bU$. 
\end{proposition}

\begin{proof}
We subdivide the proof in different parts.

\textbf{Setup and notation:} Let $([a_{1}:b_{1}],\ldots,[a_{8}:b_{8}])$ be the homogeneous coordinates of $\bX_{0} := ((\PP^{1})^{8})$. Let $\cY_{0} := \bX_{0}\times(\PP^{1} \times \PP^{1})$ and let $\pi_{0} : \cY_{0} \to \bX_{0}$ be the natural projection. For $\PP^{1} \times \PP^{1}$, let $([x_{0}:x_{1}]$, $[y_{0}:y_{1}])$ be the homogeneous coordinates. 

As usual, let $\cC$ be the divisor on $\cY_{0}$ defined by
\[
	y_{0}^{2}\prod_{i=1}^{4}(b_{i}x_{0} - a_{i}x_{1})
	+ y_{1}^{2}\prod_{i=5}^{8}(b_{i}x_{0} - a_{i}x_{1}) = 0.
\]
Let $\cL_{0} := V(y_{0}),\cL_{1}:=V(y_{1}),\cB_{0} := \cC + \cL_{0} + \cL_{1}$. Then $(\cY_{0},\frac{1+\epsilon}{2}\cB_{0})$ is a family of pairs over $\bX_{0}$. Note that these pairs are stable over $\bX_{0}^{s}$. 

Recall that for $I \in {[8] \choose 4}$, $\Delta_{I, I^{c}} := \Delta_{I} \cap \Delta_{I^{c}}$ is a connected component of the locus of strictly semi-stable points with closed orbits. Consult \S\,\ref{ssec:P18} for the detail. We say that $\Delta_{I, I^{c}}$ or the exceptional divisor $E_{I, I^{c}}$ is of \textsf{type a} if $\Delta_{I, I^{c}}$ parametrizes \textsf{type a} point configurations. In the same way we define \textsf{type b} and \textsf{type c} components. 

Let $\bX_{1} \to \bX_{0}^{ss}$ be the blow-up along $\cup \Delta_{I, I^{c}}$ and let $\rho_{1}$ be the composition $\bX_{1}^{s} \hookrightarrow \bX_{1} \to \bX_{0}^{ss}$. Let $(\cY_{1},\frac{1+\epsilon}{2}\cB_{1})$ be the pulled-back family over $\bX_{1}^{s}$, that is, $\cY_{1} :=  \bX_{1}^{s}\times(\PP^{1} \times \PP^{1})$ and $\cB_{1} := \cC_{1} + \cL_{0, 1} + \cL_{1, 1} := (\rho_{1}\times \mathrm{id})^{*}(\cC + \cL_{0} + \cL_{1})|_{\bX_0^{ss}\times(\PP^1\times\PP^1)}$. Let $\pi_{1} : \cY_{1} \to \bX_{1}^{s}$ be the natural projection. 

\textbf{Main idea:} The claimed extension $(\widetilde{\cY}, \frac{1+\epsilon}{2}\widetilde{\cB}) \to \bX_{1}^{s}$ is obtained from $(\cY_{1},\frac{1+\epsilon}{2}\cB_{1}) \to \bX_{1}^{s}$ by first appropriately blowing-up $\cY_{1}$, and then applying the relative minimal model program. It is clear that the fibers of $(\cY_{1},\frac{1+\epsilon}{2}\cB_{1}) \to \bX_{1}^{s}$ that are not stable lie above the divisors $E_{I,I^c}^s$. In what follows, we blow-up appropriate sub-loci of $\pi_1^{-1}(E_{I,I^c}^s)$ mimicking in a relative setting the blow-ups described in \S\,\ref{explicitcalculationofstablereplacement}.

\textbf{\textsf{Type a} modification:} For a \textsf{type a} divisor $E_{I, I^{c}}$, let $\cZ_{I}^{a} := \pi_{1}^{-1}(E_{I, I^{c}}^{s}) \cap V(b_{i}x_{0} - a_{i}x_{1}\mid i\in I)$. Note that we intentionally abuse our notation: it is understood that the vanishing locus $V(b_{i}x_{0} - a_{i}x_{1}\mid i\in I)\subseteq\cY_0$ is pulled-back to $\cY_1$ under the appropriate morphism. This is a smooth codimension two subvariety of $\cY_{1}$, which is a $\PP^{1}$-bundle over $E_{I, I^{c}}^{s}$. Note that $\cZ_{I^{c}}^{a}$ is also in $\pi_{1}^{-1}(E_{I, I^{c}}^{s})$ but it is disjoint from $\cZ_{I}^{a}$. Let $\cY_{2} \to \cY_{1}$ be the blow-up along $\cup(\cZ_{I}^{a} \sqcup \cZ_{I^{c}}^{a})$, where the union is over all \textsf{type a} divisors $E_{I, I^{c}}$. Fiberwisely, this is exactly the blow-up of the two double lines as in Example~\ref{typeasecondexample}. Let $\cC_{2}$ (resp. $\cL_{i, 2}$) be the proper transform of $\cC_{1}$ (resp. $\cL_{i, 1}$) and $\cB_{2} := \cC_{2} + \cL_{0, 2} + \cL_{1, 2}$. Let $\pi_{2}$ be the composition $\cY_{2} \to \cY_{1} \to \bX_{1}^{s}$. Then $\pi_{2} : (\cY_{2}, \frac{1+\epsilon}{2}\cB_{2}) \to \bX_{1}^{s}$ is a flat family of pairs. Over a \textsf{type a} divisor $E_{I, I^{c}}^{s}$, each fiber of $(\cY_{2}, \frac{1+\epsilon}{2}\cB_{2})$ is semi-log canonical and it is isomorphic to the pair in Figure \ref{fig:ssreplacementtypea}. 

\textbf{\textsf{Type b} modification:} Over a \textsf{type b} divisor $E_{I, I^{c}}$, we may assume that $|I \cap \{1,\ldots,4\}| = 3$ (hence $|I^{c} \cap \{5, 6, 7, 8\}| = 3$). Let $\cZ_{I}^{b} := \pi_{2}^{-1}(E_{I, I^{c}}^{s}) \cap \cL_{1} \cap V(b_{i}x_{0} - a_{i}x_{1}\mid i \in I)$ and let $\cZ_{I^{c}}^{b} := \pi_{2}^{-1}(E_{I, I^{c}}^{s}) \cap \cL_{0} \cap V(b_{i}x_{0} - a_{i}x_{1}\mid i \in I^{c})$. Then $\cZ_{I}^{b}$ and $\cZ_{I^{c}}^{b}$ are both disjoint sections of $E_{I, I^{c}}^{s}$. Let $\cY_{3} \to \cY_{2}$ be the blow-up along $\cup (\cZ_{I}^{b} \sqcup \cZ_{I^{c}}^{b})$ for all \textsf{type b} divisors $E_{I, I^{c}}^{s}$. Let $\cC_{3}$ (resp. $\cL_{i, 3}$) be the proper transform of $\cC_{2}$ (resp. $\cL_{i, 2}$) and let $\cB_{3} := \cC_{3} + \cL_{0, 3} + \cL_{1, 3}$. Let $\pi_{3}$ be the composition $\cY_{3} \to \cY_{2} \to \bX_{1}^{s}$. Then $\pi_{3} : (\cY_{3}, \frac{1+\epsilon}{2}\cB_{3}) \to \bX_{1}^{s}
$ is a family of pairs and over a \textsf{type b} divisor, each fiber of $(\cY_{3}, \frac{1+\epsilon}{2}\cB_{3})$ is semi-log canonical and it is isomorphic to the pair in Figure \ref{fig:ssreplacementtypeb}. 

\textbf{\textsf{Type c} modification:} Finally, over the \textsf{type c} divisor $E_{I, I^{c}}$, we may assume that $I = \{1, 2, 3, 4\}$. Let $\cP := \pi_{3}^{-1}(E_{I, I^{c}}^{s}) \cong E_{I, I^{c}}^{s} \times \PP^{1} \times \PP^{1}$. Let $\cZ_{I}^{c} := \pi_{3}^{-1}(E_{I, I^{c}}^{s}) \cap \cL_{1} \cap V(b_{i}x_{0} - a_{i}x_{1}\mid i\in I)$ and $\cZ_{I^{c}}^{c} := \pi_{3}^{-1}(E_{I, I^{c}}^{s}) \cap \cL_{0} \cap V(b_{i}x_{0} - a_{i}x_{1}\mid i \in I^{c})$. Then $\cZ_{I}^{c}$ and $\cZ_{I^{c}}^{c}$ are disjoint sections of $E_{I, I^{c}}^{s}$. Let $\cY_{4}' \to \cY_{3}$ be the weighted blow-up along $\cup (\cZ_{I}^{c} \sqcup \cZ_{I^{c}}^{c})$ where the normal subbundles $\cN_{\cZ_{I}^{c}/\cL_{1, 3}}$ and $\cN_{\cZ_{I^{c}}^{c}/\cL_{0, 3}}$ have weight two. Let $\cC_{4}'$ (resp. $\cL_{i, 4}'$) be the proper transform of $\cC_{3}$ (resp. $\cL_{i, 3}$). Set $\cB_{4}' := \cC_{4}' + \cL_{0, 4}' + \cL_{1, 4}'$. Let $\cP'$ be the proper transform of $\cP$. 

Let $\cW_{i} := \cL_{i, 4}' \cap \cP'$, which is a smooth codimension two subvariety of $\cY_{4}'$. Note that $\cW_{0}$ and $\cW_{1}$ are disjoint. Let $\cY_{4}'' \to \cY_{4}'$ be the blow-up along $\cW_{0} \sqcup \cW_{1}$. The two exceptional divisors are denoted by $E_{0}$ and $E_{1}$. Let $\cC_{4}''$, $\cL_{i, 4}''$ be the proper transforms of $\cC_{4}'$, $\cL_{i, 4}'$ respectively. Set $\cB_{4}'' := \cC_{4}'' + \cL_{0, 4}'' + \cL_{1, 4}''$. Note that $\cC_{4}'' \cong \cC_{4}'$ because $\cC_{4}'$ and $\cW_{i}$ are disjoint. 

Let $\cV_{i} := \cL_{i, 4}'' \cap E_{i}$, which is a smooth codimension two subvariety of $\cY_{4}''$. $\cV_{0}$ is disjoint from $\cV_{1}$. Let $\cY_{4} \to \cY_{4}''$ be the blow-up along $\cV_{0} \sqcup \cV_{1}$. Let $\cC_{4}$, $\cL_{i, 4}$ be the proper transforms of $\cC_{4}''$, $\cL_{i, 4}''$ respectively and let $\cB_{4} := \cC_{4} + \cL_{0, 4} + \cL_{1, 4}$. The family of pairs $\pi : (\cY_{4}, \frac{1+\epsilon}{2}\cB_{4}) \to \bX_{1}^{s}$ is semi-log canonical over the \textsf{type c} divisor. Over $E_{I, I^{c}}^{s}$, the fiber of $\cY_{4}$ has seven irreducible components. There is a `central' component, and two `tails' consisting of three irreducible components whose configurations are the same to the top three components of the fourth step in Figure \ref{case5}. 

\textbf{Contraction:} Consider the resulting family of pairs $\pi : (\cY_{4}, \frac{1+\epsilon}{2}\cB_{4}) \to \bX_{1}^{s}$. For each $x \in \bX_{1}^{s}$, the fiber $(\cY_{4 x}, \frac{1+\epsilon}{2}\cB_{4 x})$ is either irreducible stable pair or one of semi-stable pairs described in Figures~\ref{fig:ssreplacementtypea}, \ref{fig:ssreplacementtypeb}, and \ref{fig:sslimit6}. In any case, each fiber is a semi-log canonical pair and $K_{\cY_{4 x}} + \frac{1+\epsilon}{2}\cB_{4 x}$ is a nef and big divisor. Thus, by \cite[Theorem~1.10]{Fuj14}, $\rH^{1}(\cY_{4 x},\cO(m(K_{\cY_{4 x}}+\frac{1+\epsilon}{2}\cB_{4 x}))) = 0$ for $m \gg 0$. Moreover, by \cite[Theorem~1.16]{Fuj14}, $K_{\cY_{4 x}}+ \frac{1+\epsilon}{2}\cB_{4 x}$ is semi-ample. Then by the standard cohomology and base change \cite[Theorem~12.11]{Har77}, $\pi_{*}\cO(m(K_{\cY_{4}/\bX_{1}^{s}}+\frac{1+\epsilon}{2}\cB_{4}))$ is locally free and we obtain a new family of varieties 
\[
	\widetilde{\cY} := \mathrm{Proj}\; \bigoplus_{m \ge 0}\pi_{*}\cO\left(m\left(K_{\cY_{4}/\bX_{1}^{s}}+\frac{1+\epsilon}{2}\cB_{4}\right)\right)\to \bX_{1}^{s}, 
\]
and there is a contraction $\bX_{1}^{s}$-morphism $\cY_{4} \to \widetilde{\cY}$. By taking the push-forward of $\cB_{4}$, we obtain $(\widetilde{\cY}, \frac{1+\epsilon}{2}\widetilde{\cB}) \to \bX_{1}^{s}$. The resulting morphism $\widetilde{\cY} \to \bX_{1}^{s}$ is flat by \cite[Lemma~10.12]{HKT09}. Then by \cite[Theorem~4.3]{Kol18}, $(\widetilde{\cY}, \frac{1+\epsilon}{2}\widetilde{\cB}) \to \bX_{1}^{s}$ is a well-defined family of pairs. Finally, Lemmas~\ref{lem:typealc}, \ref{lem:typeblc}, and \ref{lem:typeclc} tell us that each fiber of $(\widetilde{\cY}, \frac{1+\epsilon}{2}\widetilde{\cB}) \to \bX_{1}^{s}$ is stable.
\end{proof}

\begin{proof}[Proof of Theorem \ref{thm:mainthm}]
In Proposition~\ref{compactificationsarethesame} we showed that $\overline{\bK}\cong\overline{\bJ}$, so we focus on the latter. By Proposition~\ref{extendedfamilyonkirwan'sblowup} we have a family of stable pairs over $\bX_{1}^{s}$ for the functor $\overline{\cJ}'$ in Definition~\ref{stackforbasestablepairs}, hence there is a functorial morphism $f : \bX_{1}^{s} \to \overline{\bJ}$. Since $\bU$ is an open dense subset of $\bX_{1}^{s}$, the image of $f$ is precisely $\overline{\bJ}$. Clearly the morphism $f$ is $\SL_{2}$-invariant, so there is a quotient morphism $\bar{f} : \bP = \bX_{1}^{s}/\SL_{2} \to \overline{\bJ}$. This map is also $H$-invariant, thus we can obtain yet another quotient map $\widetilde{f} : \bP/H \to \overline{\bJ}$ which we show is an isomorphism. 

Since $\widetilde{f}$ is a birational morphism between normal varieties, it is enough to show it is finite, or equivalently that $\bar{f}$ is finite. Given an exceptional divisor $E_{I,I^c}$ of the blow-up $\bX_1\rightarrow\bX_0^{ss}$, denote by $\overline{E}_{I,I^c}$ the quotient $E_{I,I^c}^s/\SL_2$. It is sufficient to show that $\overline{E}_{I,I^c}$, which is isomorphic to $\PP^2\times\PP^2$ by Remark \ref{rmk:modulimeaningofP}, is not contracted by $\bar{f}$. This is explained in Lemma~\ref{finaltechnicaldetail} below.
\end{proof}

\begin{lemma}
\label{finaltechnicaldetail}
With the notation introduced in the proof of Theorem~\ref{thm:mainthm}, the divisors $\overline{E}_{I,I^c}\subseteq\bP$ are not contracted by $\bar{f}$.
\end{lemma}

\begin{proof}
If $\overline{E}_{I,I^c}\cong\PP^2\times\PP^2$ is contracted, then at least one of the two copies of $\PP^2$, say the first component, is contracted to a point. We show that we can find two points $(\bar{p}_1,\bar{q}),(\bar{p}_2,\bar{q})\in \overline{E}_{I,I^c}$ with $\bar{p}_1\neq\bar{p}_2$ parametrizing non-isomorphic stable pairs, obtaining a contradiction. There are three cases to consider corresponding to the type of $\overline{E}_{I,I^c}$, which we define to be equal to the type of $E_{I,I^c}$.

Assume $\overline{E}_{I,I^c}$ is of \textsf{type a}. Up to relabeling, we may assume $I=\{1,2,5,6\},I^c=\{3,4,7,8\}$. The stable pair parametrized by a point in $E_{I,I^c}^s$ has two irreducible components isomorphic to $\PP^1\times\PP^1$ (see Example~\ref{typeasecondexample}). Consider the irreducible component with divisor in the form $C_0+L_0+L_1$, where $C_0$ is given by
\begin{equation*}
	y_{0}^{2}(x_{0} - t\lambda_{1})(x_{0} - t\lambda_{2})
	\lambda_{3}\lambda_{4} + 
	y_{1}^{2}(x_{0} - t\lambda_{5})(x_{0} -t\lambda_{6})
	\lambda_{7}\lambda_{8} = 0.
\end{equation*}
Recall from Remark~\ref{reasoningbehindrestrictionsonlambdas} that $\lambda_3\lambda_4\lambda_7\lambda_8\neq0$ and at least one of $\lambda_1,\lambda_2,\lambda_5,\lambda_6$ is different from the others. So pick any point $(p_1,q)\in E_{I,I^c}^s$ such that the corresponding $\lambda_1,\lambda_2,\lambda_5,\lambda_6$ are distinct. Consider the projection from $\PP^1\times\PP^1$ on the $[x_0:t]$ coordinate. The images of the four points $C_0\cap L_0,C_0\cap L_1$ are $[\lambda_1:1],[\lambda_2:1],[\lambda_5:1],[\lambda_6:1]$, which are distinct points on $\PP^1$. Denote by $\beta$ their cross-ratio. Choose $[\mu_1:1],[\mu_2:1],[\mu_5:1],[\mu_6:1]$ distinct points on $\PP^1$ such that the corresponding cross-ratio is different from $\beta$. Then the stable pair obtained by replacing $\lambda_1,\lambda_2,\lambda_5,\lambda_6$ with $\mu_1,\mu_2,\mu_5,\mu_6$ respectively and keeping $\lambda_3,\lambda_4,\lambda_7,\lambda_8$ unchanged is parametrized by a point $(p_2,q)\in E_{I,I^c}^s$ with $p_1\neq p_2$. The images $(\bar{p}_1,\bar{q}),(\bar{p}_2,\bar{q})$ in $\overline{E}_{I,I^c}$ are also distinct because the cross-ratio is $\SL_2$-invariant, showing what we needed.

The cases of \textsf{type b} and \textsf{type c} are handled similarly, but with the following differences. For \textsf{type b}, given a stable pair parametrized by $E_{I,I^c}^s$, one can consider the cross-ratio of the four points on $L_1$ given by $[\lambda_1:1],[\lambda_2:1],[\lambda_3:1],[1:0]$, where the last point is the intersection of $L_1$ with the conductor divisor. For \textsf{type c}, look at $[\lambda_1:1],[\lambda_2:1],[\lambda_3:1],[\lambda_4:1]$ on the curve $T$ in Figure~\ref{case5}.
\end{proof}

\bibliographystyle{alpha}

\end{document}